\documentclass[11pt]{amsart}
\usepackage{graphicx,amssymb,amsmath,amsthm}
\usepackage{amsfonts}
\usepackage{enumerate}
\usepackage{dsfont}
\usepackage{cite}
\usepackage[colorlinks, citecolor=red]{hyperref}
\usepackage{mathrsfs}
\usepackage{epsfig}
\usepackage{lscape}
\usepackage{subfigure}
\usepackage{epstopdf}
\usepackage{caption}
\usepackage{algorithm}
\usepackage{algpseudocode}
\usepackage{multirow}
\usepackage{geometry}
\usepackage{paralist}
\usepackage{enumerate}
\usepackage{url}
\usepackage{tikz}

\newtheorem{definition}{Definition}[section]

\newtheorem{theorem}[definition]{Theorem}
\newtheorem{lemma}[definition]{Lemma}

\newtheorem{remark}[definition]{Remark}

\date{}

\begin{document}
\baselineskip 18pt
\title[RDR for linear systems]{
Randomized Douglas-Rachford methods for linear systems: Improved  accuracy and efficiency
}

\author{Deren Han}
\address{LMIB of the Ministry of Education, School of Mathematical Sciences, Beihang University, Beijing, 100191, China. }
\email{handr@buaa.edu.cn}

\author{Yansheng Su}
\address{School of Mathematical Sciences, Beihang University, Beijing, 100191, China.}
\email{suyansheng@buaa.edu.cn}

\author{Jiaxin Xie}
\address{LMIB of the Ministry of Education, School of Mathematical Sciences, Beihang University, Beijing, 100191, China. }
\email{xiejx@buaa.edu.cn}

\begin{abstract}
The Douglas-Rachford (DR) method is a widely used method for finding a point in the intersection of two closed convex sets (feasibility problem).
However,  the method converges weakly and the associated rate of convergence is hard to analyze in general.
In addition, the direct extension of the DR method for solving more-than-two-sets feasibility problems, called the $r$-sets-DR method, is not necessarily convergent.
To improve the efficiency of the optimization algorithms,  the introduction of randomization and the momentum technique has attracted increasing attention.
In this paper, we propose the randomized $r$-sets-DR (RrDR) method for solving the feasibility problem derived from linear systems, showing the benefit of the randomization as it brings linear convergence in expectation to the otherwise divergent $r$-sets-DR method. Furthermore, the convergence rate does not depend on the dimension of the coefficient matrix.
We also study RrDR with heavy ball momentum and establish its accelerated rate.
Numerical experiments are provided to confirm our results and demonstrate the notable improvements in  accuracy and efficiency of the DR method, brought by the randomization and the momentum technique.
\end{abstract}

\maketitle

\let\thefootnote\relax\footnotetext{Key words: 	Douglas-Rachford; Randomization; Heavy ball momentum; Convergence rate; Linear systems; Kaczmarz method}

\let\thefootnote\relax\footnotetext{Mathematics subject classification (2020): 90C25, 65F10, 65F20, 68W20}

	\section{Introduction}\label{sec:intro}
\subsection{Problem setup}

Consider the large-scale system of linear equations
\begin{equation}
	\label{main-prob}
	Ax=b,
\end{equation}
where $A\in\mathbb{R}^{m\times n}$ and $b\in\mathbb{R}^m$.
The problem of solving the linear system \eqref{main-prob} arises in various fields of science and engineering, such as
optimal control \cite{Pat17}, machine learning \cite{Cha08}, signal processing \cite{Byr04},  and partial differential equations \cite{Ols14}. Throughout this paper, we assume that this linear system is consistent and
refer to a certain solution of \eqref{main-prob} as $x^*$.
% use $x^*$ to denote a certain solution of \eqref{main-prob}.

Let $a_1^\top,a_2^\top,\ldots,a_m^\top$ denote the rows of $A$ and let $b=(b_1,\ldots,b_m)^\top$. Then \eqref{main-prob} can be rewritten as the following \emph{feasibility problem}
\begin{equation}
	\label{feaprob}
	\mbox{Find} \ x^*\in C = \mathop{\cap}\limits_{i=1}^m C_i, \quad \mbox{where} \ C_i:=\{x:\langle a_i,x\rangle=b_i\}.
\end{equation}
A classic approach to solving \eqref{feaprob} is the \emph{projection method}.
However, it can be costly to compute the projection onto the intersection $C$.
So a more practical strategy is to successively project the current iterate onto a single feasible set $C_i$ at each iteration, where $ C_i $ is chosen in a certain manner.
There are numerous such methods, for instance, the Dykstra's method \cite{Dyk83}, the von Neumann method \cite{Bau93}, and the Douglas-Rachford (DR) method \cite{Dou56}.
In this paper, we focus on the DR method.

\subsection{ Douglas-Rachford method}
\label{sect:DR}
The Douglas-Rachford %(DR)
method \cite{Lin20} is a notable splitting approach for finding zeros of the sum of maximal monotone operators.
%It was originally introduced in \cite{Dou56} to numerically solve certain types of heat equations.
It has already been applied to various optimization problems
% that arise in signal processing and other applications, where the objective function is a sum of proper closed convex functions; see \cite{Lin20,Ara14,Li16,Eck92}.
whose objective function is a sum of proper closed convex functions; see \cite{Lin20,Ara14,Li16,Eck92}.
%Since the feasibility problem is a special case where the maximal monotone operators are normal cones, thus the DR method can also be applied.
Since feasibility problems are special cases where the operators are normal cones, it implies that the DR method can be applied to them.
For any closed set $C_i$, let $P_{C_i}$ denote the orthogonal \emph{projection} operator onto $C_i$ and $R_{C_i}:=2P_{C_i}-I$ denote the \emph{reflection} operator over $C_i$, where $I$ denotes the identity operator.
Specifically, when $C_i=\{x:\langle a_i,x\rangle=b_i\}$ is a hyperplane, for any $x\in \mathbb{R}^n$ we have
$$P_{C_i}(x)=x-\frac{\langle a_i,x\rangle-b_i}{\|a_i\|^2_2}a_i, \ \ \mbox{and} \ \ R_{C_i}(x)=x-2\frac{\langle a_i,x\rangle-b_i}{\|a_i\|^2_2}a_i.$$
The canonical DR method is restricted to dealing with only two sets, namely, finding a feasible point  in the intersection of $C_1$ and $C_2$.  Starting from a proper $x^0$, %the iteration scheme of the method  reads as
it iterates with the format
%In order to find a feasible point in the intersection of $C_1$ and $C_2$, the method starts from a proper $x^0$ and iterates as
$$
x^{k+1}=\frac{1}{2}\left(I+R_{C_2}R_{C_1}\right)(x^k),
$$
which is illustrated in Figure \ref{figueDRmethod}.
%That is why
For this reason, the DR scheme is also known as \emph{reflect-reflect-average} \cite{Bor11} or \emph{averaged alternating reflections} \cite{Bau04} in the literature.

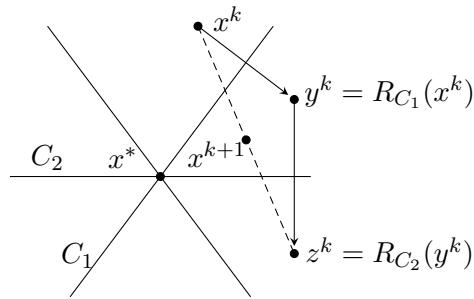
\begin{figure}[hptb]
	\centering
	\begin{tikzpicture}%[>=Stealth]
		\draw (-2,0) -- (2,0);
		\draw (-1.5,2) -- (1.2,-1.6);
		\draw (-1.2,-1.6) -- (1.5,2);
		
		\filldraw (0,0) circle [radius=1.5pt]
		(0.5,2) circle [radius=1.5pt]
		(1.78,1.025) circle [radius=1.5pt]
		(1.78,-1.025) circle [radius=1.5pt]
		(1.14,0.4875) circle [radius=1.5pt];
		
		\draw (-0.5,0.25) node {$x^*$}
		(0.9,2.1) node {$x^k$}
		(3.05,1.125) node {$y^k=R_{C_1}(x^k)$}
		(3.05,-1.025) node {$z^k=R_{C_2}(y^k)$}
		(0.75,0.3) node {$x^{k+1}$};
		\draw (-1.5,0.25) node {$C_2$}
		(-1.1,-1.05) node {$C_1$};
		
		\draw [-stealth] (0.5,2) -- (1.7,1.085);
		\draw [-stealth] (1.78,1.025) -- (1.78,-0.925);
		\draw [dash pattern=on 3pt off 2pt] (0.5,2) -- (1.78,-1.025);
	\end{tikzpicture}
	\caption{One step of the  Douglas-Rachford method: Reflect-reflect-average; $x^{k+1}=\frac{1}{2}(x^k+z^k)=\frac{1}{2}(I+R_{C_2}R_{C_1})(x^k)$. }
	\label{figueDRmethod}
\end{figure}

However, when it comes to the $m$-sets $(m>2)$ convex feasibility problem,  the direct extension of the canonical Douglas-Rachford method $ x^{k+1}=\frac{1}{2}(I+R_{C_m}\cdots R_{C_2}$ $R_{C_1})(x^k) $ may fail even in simple instances. See an example from Artacho, Borwein, and Tam  \cite{Ara14} in Figure \ref{figue3sets}.
To overcome this problem, Borwein and Tam \cite{Bor14} devised the cyclic DR method
by using only two sets at a time
\begin{equation}\label{cyc-DR}
	x^{k+1}=\frac{1}{2}(I+R_{C_{t_{k+1}}}R_{C_{t_{k}}})(x^k),
\end{equation}
where $t_k=(k \ \operatorname{mod} \ m)+1$. This pattern can be extended by employing more sets at each time, and such ideas are detailed as the \textit{cyclic $r$-sets-DR method} \cite{aragon2019cyclic,censor2016new}.
While the conditions for weak convergence of the canonical DR method and the cyclic $r$-sets-DR method have been established, it remains difficult to obtain effective theoretical estimates of their convergence rate. Due to the complexity of the required computations,  existing estimates are not comparable to those of other state-of-the-art iterative methods \cite{Bui20, Bau04, Bau14, Bau15Lin, Lio79, Gal05}. Given this situation, we intend to introduce modern optimization techniques to improve the properties of the DR method.
%%%%%%%%%%%
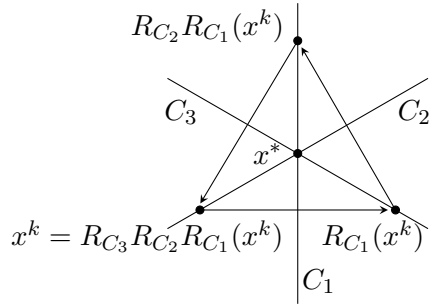
\begin{figure}[hptb]
	\centering
	\begin{tikzpicture}%[>=Stealth]
		\draw (0,2) -- (0,-2);
		\draw (-1.732,1) -- (1.732,-1);
		\draw (-1.732,-1) -- (1.732,1);
		
		\filldraw (0,0) circle [radius=1.5pt]
		(0,1.5) circle [radius=1.5pt]
		(-1.3,-0.7516) circle [radius=1.5pt]
		(1.3,-0.7516) circle [radius=1.5pt];
		
		\draw (-0.4,0) node {$x^*$}
		(-2,-1.1) node {$x^k=R_{C_3}R_{C_2}R_{C_1}(x^k)$}
		(1.0,-1.1) node {$R_{C_1}(x^k)$}
		(-1.2,1.7) node {$R_{C_2}R_{C_1}(x^k)$};
		\draw (0.28,-1.7) node {$C_1$}
		(1.55,0.55) node {$C_2$}
		(-1.55,0.55) node {$C_3$};
		
		\draw [-stealth] (0,1.5) -- (-1.295,-0.665);
		\draw [-stealth] (-1.3,-0.7516) -- (1.2,-0.7516);
		\draw [-stealth] (1.3,-0.7516) -- (0.05,1.4134);
	\end{tikzpicture}
	\caption{ Failure of the 3-sets-DR iteration: The iteration $x^k:=\big(\frac{1}{2}I+\frac{1}{2}R_{C_3}R_{C_2}R_{C_1}\big)^k(x^0)$ may cycle.}
	\label{figue3sets}
\end{figure}

\subsection{Our contribution}

In this paper, we investigate the  DR  method with randomization for solving feasibility problems derived from linear systems. The main contributions of this work are as follows.

\begin{itemize}
	\item[1.] %For one thing, witnessing the reliable performance of randomization in tackling large-scale problems, we incorporate it into DR method and propose our randomized $r$-sets-DR (RrDR) method for solving the linear system $Ax=b$.
	We introduce the randomization technique to the $r$-sets-DR (RrDR) method for solving the feasibility problem \eqref{feaprob} and demonstrate that this approach is effective in simplifying the analysis of the $r$-sets-DR method and endows the otherwise divergent $r$-sets-DR method with a linear convergence.
	Specifically, we prove that the expected norm of the error $ \mathbb{E} [ \|x^k-x^* \|^2_2] $ of RrDR converges linearly, with the convergence rate depending only on the singular values of $ A $ and the relaxation parameter of the algorithm, but not on the size of $ A $.
	\item[2.] We then focus on a variant of the RrDR method with momentum (mRrDR), which is inspired by the success of Polyak's heavy ball momentum method \cite{polyak1964some,loizou2020momentum,morshed2020stochastic}.
	Although  the expected norm of the error $ \mathbb{E} [ \|x^k-x^* \|^2_2] $ of mRrDR also converges linearly, we find that its convergence rate is weaker than that of the RrDR method. Therefore, we also consider the norm of the expected error $\|\mathbb{E}[x^k-x^*]\|_2^2$, in terms of which mRrDR shows superiority over RrDR.
	%However, its convergence is weaker than that of $ \mathbb{E} [ \|x^k-x^* \|^2_2] $. See Section \ref{sect-rrd-c} for detailed discussions.
	To the best of our knowledge, this is the first study to investigate the momentum variants of the $r$-sets-DR method.
	\item[3.] Furthermore, we also compare our DR-originated method with other state-of-the-art iterative methods for linear systems. Moveover, we even show the superiority of the mRrDR over the  built-in function of {\sc Matlab} {\tt pinv} and {\tt lsqminnorm} when the number of the rows of $A$ is sufficiently larger than the number of columns.
\end{itemize}

\subsection{Notations}
\label{sect1-3}
We use $\mathbb{Z}_{+}$ to denote the set of positive integers.
For any random variables $\xi$ and $\zeta$, we use $\mathbb{E}[\xi]$ and $\mathbb{E}[\xi \lvert \zeta = \zeta_0]$ to denote the expectation of $\xi$ and the conditional expectation of $\xi$ given $\zeta = \zeta_0$.
For vector $x\in\mathbb{R}^n$, we use $x_i,x^\top$, and $\|x\|_2$ to denote the $i$-th entry, the transpose, and the Euclidean norm of $x$, respectively.
For matrix $A\in\mathbb{R}^{m\times n}$, we use $a_i$, $A_j$, $A^\top$, $A^\dagger$, $ \|A\|_2 $, $ \|A\|_F $, $\mbox{Row}(A)$, and $\mbox{Range}(A)$ to denote the $i$-th row, $j$-th column, the transpose, the Moore-Penrose pseudoinverse, the spectual norm, the Frobenius norm, the row space, and the range space of $ A $, repectively.
We use $A=U\Sigma V^\top$ to denote the singular value decomposition (SVD) of $A$, where $U\in\mathbb{R}^{m\times m}$, $\Sigma\in\mathbb{R}^{m\times n}$, and $V\in\mathbb{R}^{n\times n}$.
The nonzero singular values of $A$ are $\sigma_1(A)\geq\sigma_2(A)\geq\ldots\geq\sigma_{t}(A):=\sigma_{\min}(A)>0$, where $t$ is the rank of $A$
and  $\sigma_{\min}(A)$  is  the smallest nonzero singular value of $A$.
% We note the fact that $\|A\|_2=\sigma_{1}(A)$ and $\|A\|_F=(\sum_{i=1}^r \sigma_i(A)^2)^{\frac{1}{2}}$.

Throughout this paper, we use $x^*$ to denote a certain solution of the linear system \eqref{main-prob}, and for any $x^0\in\mathbb{R}^n$, we set
$x_0^*:=A^\dagger b+(I-A^\dagger A)x^0
\
\text{and}
\ x_{LN}^*:=A^\dagger b.$
We mention that $x_0^*$
is the orthogonal projection of $x^0$ onto the set
$
\{x\in\mathbb{R}^n| A x= b\},
$
and $x_{LN}^*$ is the least-norm solution of the linear system.
%In this paper, by ``momentum'' we refer to the heavy ball technique originally developed by Polyak \cite{polyak1964some} to accelerate the convergence rate of gradient-type methods.

\subsection{Organization}
The remainder of the paper is organized as follows. %In Section \ref{sect1-3}, we will give some notations and in Section \ref{sec:rw}, we will give a review of the related work.
In Section \ref{sec:rw} we will give a review of the related work.
Section \ref{sec:rrdr} and Section \ref{sec:mrrdr} describe randomized $r$-sets-Douglas-Rachford method and its momentum variant, respectively. Section \ref{sec:ne} reports the mentioned numerical experiments and Section \ref{sec:conc} concludes the paper. Proofs of all main results are provided in the appendix.

\section{Related work}\label{sec:rw}
%\subsection{Kaczmarz method}
One of the most widely used projection solvers for \eqref{main-prob} is the \emph{Kaczmarz} method \cite{Kac37}, which can be recognized as a special kind of the Dykstra's method. Starting from  $x^0\in\mathbb{R}^n$, the canonical Kaczmarz method constructs $x^{k+1}$ by
%%%%%%ADV: i or ik
$$
x^{k+1}=P_{C_{i_k}}(x^k)=x^k-\frac{\langle a_{i_k},x^k\rangle-b_{i_k}}{\|a_{i_k}\|^2_2}a_{i_k},
$$
where $ {i_k} $ is cyclically selected from $ \{1,\cdots,m \} $. The sequence $\{x^k\}_{k=0}^\infty$ converges to $x_0^*$ but the convergence rate is hard to obtain. In the seminal paper \cite{Str09}, Strohmer and Vershynin first investigated the randomized Kaczmarz (RK) method. Specifically, they proved that if $ i_k $ is selected with probability $ \mbox{Pr}(i_k=i)=\frac{\|a_{i}\|_2^2}{\|A\|_F^2}$, the method converges linearly
%\begin{equation}\label{conv-RK}
$$	\mathbb{E} [ \|x^k-x_0^*\|^2_2]   \leq \left(1-\frac{\sigma^2_{\min}(A)}{\|A\|_F^2}\right)^k\|x^0-x_0^*\|_2^2.$$
%\end{equation}
Subsequently, there has been a significant amount of work on the development of Kaczmarz-type methods, with references available in \cite{Bai18Gre,Du20Ran,Du20Kac,Gow15,Gow19,Hef15,Jia17,Lev10,Ste20Ran,Ste20Wei,yuan2022adaptively,Raz19}.
Among these methods, we pay extra attention to the \textit{reflection Kaczmarz method} studied by Steinerberger in \cite{Ste20Sur}, which constructs $x^{k+1}$ via
$$
x^{k+1}=R_{C_{i_k}}(x^k)=x^k-2\frac{\langle a_{i_k},x^k\rangle-b_{i_k}}{\|a_{i_k}\|_2^2}a_{i_k}.
$$
Although the generated sequence $\{x^k\}_{k=0}^\infty$ does not converge to $x^*$ as their distance $\|x^k-x^*\|_2 \equiv \|x^0-x^*\|_2$ remains constant, Steinerberger \cite{Ste20Sur} proved the sublinear convergence of
the average $\frac{1}{k}\sum_{i=1}^kx^i$. Moreover, the author noted that with a particular restart strategy, the algorithm can reach the same complexity as the RK method. In fact, our work is inspired by such reflection approaches \cite{Ste20Sur,Sha21}, while instead of using the orthogonal projections or reflections only, we study the DR method which incorporates the reflection-average approach.

Recently, Hu and Cai \cite{Hu20} studied the canonical randomized Douglas-Rachford (RDR) method in a simple case where $r=2$.
They proved that $\mathbb{E}[x^k]\to x^*$ as $k\to\infty$, however, without convergence rates analysis.
We also note that such convergence does not lead to $\mathbb{E} [ \|x^k-x^*\|^2_2 ]\to 0$.
As an exemplification \cite{Wri20}, consider $x^k=x^*+r^k$, where $r^k$ are drawn i.i.d. from $N(0, I/\sqrt{n})$.
Such a sequence satisfies $\mathbb{E}[x^k]= x^*$, while $\mathbb{E} [ \|x^k-x^*\|^2_2 ] \equiv 1$.

In recent years, the momentum acceleration technique has been recognized as an effective approach to improving the performance of optimization methods. For instance, the stochastic gradient descent (SGD) method \cite{robbins1951stochastic} can be incorporated with the heavy ball momentum \cite{polyak1964some}, deriving the well-known stochastic heavy ball momentum (SHBM) method \cite{sebbouh2021almost,garrigos2023handbook} with enhanced performance for solving large-scale optimization problems.
The RK method, as a special case of SGD (see Section \ref{section-RrDR}), was studied in the momentum framework \cite{loizou2020momentum,morshed2020stochastic}.
Specifically, Loizou and Richt{\'a}rik \cite{loizou2020momentum} investigated the SHBM method for solving stochastic problems that are reformulated from consistent linear systems, and established the global, non-asymptotic linear convergence rates of the proposed methods. In \cite{Gow15}, Gower and Richt\'{a}rik developed the sketch-and-project method, a versatile randomized iterative method which includes the RK algorithm as a special case, for solving consistent linear systems.
The momentum variants of the sketch-and-project method has been investigated in \cite{P2017Stochastic,loizou2021revisiting}. Enlightened by the success of the heavy ball momentum technique in these stochastic methods, we intend to employ such momentum acceleration in the randomized DR method.

We note that another popular momentum acceleration is the Nesterov's momentum \cite{nesterov1983method,nesterov2003introductory}, leading to the famous accelerated gradient descent (AGD) method \cite{beck2009fast}. Recently, variants of Nesterov's momentum has also been introduced for the acceleration of stochastic optimization algorithms \cite{lan2020first}.  In \cite{liu2016accelerated}, Liu and Wright applied the acceleration scheme of Nesterov to the RK method.
It has also been applied to the sampling Kaczmarz Motzkin (SKM) algorithm for linear feasibility problems \cite{morshed2020accelerated,morshed2021sampling}.

\section{Randomized Douglas-Rachford method}\label{sec:rrdr}
\label{section-RrDR}
In this section, we propose our randomized Douglas-Rachford method for solving linear systems and analyze its convergence property. Actually, for the canonical DR algorithm, there are a lot of modifications and relaxations in the literature. An important and useful one is the approach studied by Eckstein and Bertsekas \cite{Eck92}, known as the \emph{generalized Douglas-Rachford} (GDR) method
$$
x^{k+1}=\big((1-\alpha) I+\alpha R_{C_2}R_{C_1}\big)(x^k), \ \alpha\in(0,1).
$$
It reduces to the DR method when $\alpha=\frac{1}{2}$.
The significance of introducing the relaxation parameter $ \alpha $ is that, in general, the performance of the optimization algorithms can be practically accelerated under overrelaxation conditions ($ \alpha > \frac{1}{2} $)\cite{Str09,Nec19}.

Based on the variants of the DR method for $ m$-sets condition discussed in Section \ref{sect:DR}, we consider a more universal setting of the DR method, called the extrapolated $r$-sets-DR method. At the $ k$-th iteration, the algorithm chooses $ r $ sets $ C_{j_{k_1}}, \ldots, C_{j_{k_r}} $ and updates $ x^{k} $ via
$$
x^{k+1}=\big((1-\alpha)I+\alpha R_{C_{j_{k_r}}}R_{C_{j_{k_{r-1}}}}\ldots R_{C_{j_{k_1}}}\big)(x^k),
$$
where $\alpha\in(0,1)$ is the extrapolation or relaxation parameter. The criterion of the selection of the sets
$ C_{j_{k_1}}, \cdots, C_{j_{k_r}} $ is where the randomization is adopted.
We prove that if the set $C_i$ is selected with probability proportional to $\|a_i\|^2_2$, the method converges linearly in expectation.

A description of the iterative
procedure is presented in Algorithm \ref{r-RDRK}.
To demonstrate the algorithm in a simple and straightforward form,
the description is in terms of $a_1,\ldots, a_m$ and $b=(b_1,\ldots,b_m)^\top$ rather than the operators $R_{C_i},i=1,\ldots,m$. It can be seen  that not only the canonical RDR method is included in this framework with $ r=2 $ and $ \alpha=\frac{1}{2} $, but also the RK method can be recovered with $ r=1 $ and $ \alpha=\frac{1}{2} $.

Additionally, we assume that $\mbox{rank}(A)\geq 2$. Suppose that $ \operatorname{rank}(A) = 1 $ and $r$ is an even number, now the extrapolated $r$-sets-DR method just reflects the iteration back and forth through the same hyperplane. So the iteration sequence of the method satisfies $x^0=x^1=\ldots=x^k$ and the method fails. In fact, $\mbox{rank}(A)=1$ is a trivial case where one can get the solution by only one step of projection.

\begin{algorithm}[htpb]
	\caption{Randomized $r$-sets-Douglas-Rachford (RrDR) method \label{r-RDRK}}
	\begin{algorithmic}
		\Require
		$A\in \mathbb{R}^{m\times n}$, $b\in \mathbb{R}^m$, $r\in\mathbb{Z}_{+}$, $k=0$,  extrapolation/relaxation parameter $\alpha\in(0,1)$ and an initial $x^0\in \mathbb{R}^{n}$.
		\begin{enumerate}
			\item[1:] Set $z^{k}_0:=x^k$.
			\item[2:] {\bf for $\ell=1,\ldots,r$ do}
			\item[3:] \ \ \ Select $j_{k_{\ell}}\in\{1,\ldots,m\}$ with probability $\mbox{Pr}(\mbox{row}=j_{k_{\ell}})=\frac{\|a_{j_{k_{\ell}}}\|^2_2}{\|A\|_{F}^2}$.
			\item[4:] \ \ \ Compute
			$$
			z_{\ell}^{k}:=z_{\ell-1}^k-2\frac{\langle a_{j_{k_{\ell}}},z_{\ell-1}^k\rangle-b_{j_{k_{\ell}}}}{\|a_{j_{k_\ell}}\|^2_2}a_{j_{k_{\ell}}}.
			$$
			\item[5:] {\bf end for}
			\item[6:] Update
			$$
			x^{k+1}:=(1-\alpha) x^k+\alpha z_{r}^k.
			$$
			\item[7:] If the stopping rule is satisfied, stop and go to output. Otherwise, set $k=k+1$ and return to Step $1$.
		\end{enumerate}
		
		\Ensure
		The approximate solution $ x^k $.
	\end{algorithmic}
\end{algorithm}

Finally, we give some comments on the connection between Algorithm  \ref{r-RDRK} and  the SGD method \cite{hardt2016train,robbins1951stochastic,ma2017stochastic}.
The RK method can be viewed as SGD applied to the  following least-squares problem
\begin{equation}\label{quar-f}
	\min\limits_{x\in\mathbb{R}^n}f(x):=\frac{1}{2 m}\|A x-b\|_2^2=\frac{1}{m} \sum_{i=1}^m f_i(x),
\end{equation}
where $f_i(x)=\frac{1}{2}\left( \langle a_i,x \rangle -b_i\right)^2$. For convenience, we assume that $A \in \mathbb{R}^{m \times n}$ is normalized such that $\|a_i\|^2_2=1$.
The SGD method solves \eqref{quar-f} using unbiased estimates for the gradient of the objective function. Particularly, one can employ $\nabla f_i(x)$ since $ \mathbb{E} [\nabla f_i(x)] = \nabla f(x) $.
At the $k$-th iteration, SGD draws $\nabla f_{i_k}(x)$ and  updates $x^k$ via
\begin{equation}\label{SGD}
	x^{k+1}=x^k-\lambda_k \nabla f_{i_k}(x^k),
\end{equation}
where $\lambda_k$ is an appropriately chosen stepsize.
Since here $\nabla f_{i_k}(x^k)=\left( \langle a_{{i_k}},x^k \rangle -b_{i_k}\right)a_{i_k}$, one can recover the RK method by applying SGD  to the problem \eqref{quar-f}.
%By selecting a random row of the matrix $A$ and computing \eqref{SGD} with $\nabla f_{i_k}(x^k)=\left( a_{{i_k}}^\top x^k-b_{i_k}\right)a_{i_k}$, one can recover the RK method.
Based on this fact, we reconsider Algorithm \ref{r-RDRK} with $\alpha=\frac{1}{2}$ and $r=2$, where we have $z_1^k=x^k-2\nabla f_{j_{k_1}}(x^k)$, $z_2^k=z_1^k-2\nabla f_{j_{k_2}}(z_1^k)$, and hence
$$
x^{k+1}=x^k-\nabla f_{j_{k_1}}(x^k)-\nabla f_{j_{k_2}}\left(x^k-2\nabla f_{j_{k_1}}(x^k)\right).
$$
In fact, the last term $\nabla f_{j_{k_2}}\left(x^k-2\nabla f_{j_{k_1}}(x^k)\right)$ is known as the \emph{stochastic extragradient} (SEG) \cite{gorbunov2022stochastic}. Therefore, we state that Algorithm \ref{r-RDRK} can be regarded as a combination of SGD and SEG.

\subsection{ Convergence of iterates}
\label{sect-rrd-c}
In this subsection, we analyze the convergence properties of Algorithm \ref{r-RDRK}. The first result demonstrates that the expected norm of the error $\mathbb{E} [ \|x^k-x_{0}^*\|^2_2]$ converges linearly.% convergence in expectation.
\begin{theorem}
	\label{main-ThmrRDR}
	Suppose that  the linear system $Ax=b$ is consistent, $\alpha\in(0,1)$, and $x^0\in\mathbb{R}^n$ is an arbitrary initial vector.
	Let $x_0^*=A^{\dagger}b+(I-A^\dagger A)x^0$. Then the iteration sequence $\{x^k\}_{k=0}^{\infty}$ generated by Algorithm \ref{r-RDRK} satisfies
	$$\mathbb{E} [ \|x^k-x_{0}^*\|^2_2]\leq\left(\alpha^2+(1-\alpha)^2+2\alpha(1-\alpha)\left(1-2\frac{\sigma_{\min}^2(A)}{\|A\|^2_F}\right)^r
	\right)^k\|x^0-x_{0}^*\|^2_2.
	$$
\end{theorem}

\begin{remark}
	If the initial vector $x^0\in\mbox{Row}(A)$, then we have $x_0^*=A^{\dagger}b=x^*_{LN}$. This implies that the iteration sequence $ \{ x^k \}_{k=0}^\infty $ generated by Algorithm \ref{r-RDRK} now converges to the least-norm solution $x^*_{LN}$.  Additionally,  we note that the upper bound in Theorem \ref{main-ThmrRDR} is tight. Please refer to Remark \ref{xie-tight} for more details.
\end{remark}

\begin{remark}
	We note that the almost sure convergence of the iterates of the stochastic algorithms has been extensively studied \cite{sebbouh2021almost,combettes2015stochastic,briceno2022random}. By utilizing Lemma 2.1 in \cite{briceno2022random} and Theorem \ref{main-ThmrRDR}, it can be easily deduced that the iteration sequence $ \{ x^k \}_{k=0}^\infty $ generated by Algorithm \ref{r-RDRK}  converges almost surely to $x^*_0$.
\end{remark}

\begin{remark}
	Let $\widetilde{\alpha}:=2\alpha\in(0,2)$. If $r=1$, then Algorithm \ref{r-RDRK} becomes $x^{k+1}=x^k-\widetilde{\alpha}\frac{\langle a_{i_k},x^k\rangle-b_{i_k}}{\|a_{i_k}\|^2_2}a_{i_k}$ and Theorem \ref{main-ThmrRDR} leads to
	%	\begin{equation}\label{xie-relaxation-para}
		$$	\mathbb{E} [ \|x^k-x_{0}^*\|^2_2 ] \leq\left(1-\frac{\widetilde{\alpha}(2-\widetilde{\alpha})\sigma_{\min}^2(A)}{\|A\|^2_F}
		\right)^k\|x^0-x_{0}^*\|^2_2,$$
		%	\end{equation}
	which is exactly the conclusion obtained in \cite[Theorem 1]{Cai12} for the RK method with relaxation.
\end{remark}

We now make a comparison between our method and the RK method in terms of convergence rate. Since the computational cost of Algorithm \ref{r-RDRK} at each step is about $r$-times as expensive as that of the RK method, the comparison can be expressed as
\begin{equation}\label{rk-rrdr-xie}
	\left(1-\frac{4\alpha(1-\alpha)\sigma_{\min}^2(A)}{\|A\|^2_F}\right)^r\leq
	\alpha^2+(1-\alpha)^2+2\alpha(1-\alpha)\left(1-2\frac{\sigma_{\min}^2(A)}{\|A\|^2_F}\right)^r.
\end{equation}
Indeed, let $p=1-\frac{4\alpha(1-\alpha)\sigma_{\min}^2(A)}{\|A\|^2_F}$ and $q=1-2\frac{\sigma_{\min}^2(A)}{\|A\|^2_F}$, then for any fixed $\alpha\in(0,1)$, one can verify that
$$
g(r):=\frac{\alpha^2+(1-\alpha)^2}{p^r}+2\alpha(1-\alpha)\left(\frac{q}{p}\right)^r
$$
is monotonically increasing, i.e. $g(r)\geq g(1)=1$ so that \eqref{rk-rrdr-xie} holds.
This implies that the RK method is theoretically better than  Algorithm \ref{r-RDRK}.
Nevertheless, numerical experiments demonstrate that, with an appropriate $r>1$, Algorithm \ref{r-RDRK} is more efficient and requires fewer row-actions than the RK method (see Section \ref{subsection:cr}).

Next, let us consider the convergence of the norm of the expected error $\|\mathbb{E}[x^k-x^*_0]\|_2^2$.
\begin{theorem}
	\label{main-ThmrRDRv2}
	Suppose that  the linear system $Ax=b$ is consistent, $x^0\in\mathbb{R}^n$ is an arbitrary initial vector, and the relaxation parameter satisfies
	$$
	0<\alpha<\min\left\{1, \frac{1}{1-\left(1-2\sigma_{\max}^2(A)/\|A\|^2_F\right)^{r}}\right\}.$$ Let $x_0^*=A^{\dagger}b+(I-A^\dagger A)x^0$. Then the iteration sequence $\{x^k\}_{k=0}^{\infty}$ generated by Algorithm \ref{r-RDRK} satisfies
	$$
	\| \mathbb{E} [ x^k-x_{0}^* ] \|^2_2 \leq \left((1-\alpha)+\alpha\left(1-2\frac{\sigma_{\min}^2(A)}{\|A\|^2_F}\right)^r
	\right)^{2k}\|x^0-x_0^*\|_2^2.
	$$
\end{theorem}

Since $ (1-\alpha)+\alpha\left(1-2\sigma_{\min}^2(A)/\|A\|^2_F\right)^r \in (0,1)$, we know that $\|\mathbb{E}[x^k-x^*_0]\|_2^2\to 0$  as $k\to \infty$.
One may be confused by the similarity between the quantity $\|\mathbb{E}[x^k-x^*_0]\|_2^2$ in Theorem \ref{main-ThmrRDRv2} and $ \mathbb{E} [ \|x^k-x^*_0 \|^2_2] $ in Theorem \ref{main-ThmrRDR}. Actually, the convergence of the former is much weaker than that of the later.
Supposing $\mathbb{E}[x]$ is bounded for all $ x \in \mathbb{R}^n $, by definition, we have
$$
\mathbb{E} [ \|x-x^*_0 \|^2_2] =\left\|\mathbb{E} \left[ x-x_0^{*} \right] \right\|_2^{2}
+\mathbb{E} \left[ \|x-\mathbb{E}[x]\|_2^{2}\right],
$$
which implies that the convergence of $ \mathbb{E} [ \|x-x^*_0 \|^2_2] $ leads to that of $ \left\|\mathbb{E} \left[ x-x_0^{*} \right] \right\|_2^{2} $, but not vice versa. The reason why it is also considered is that we aim to systematically compare RrDR method with its momentum variant on both quantities in Section \ref{sec:mrrdr}.

\subsection{Convergence direction}
Inspired by the recent work of Steinerberger \cite{Ste20Ran}, in this subsection we consider the convergence direction of Algorithm \ref{r-RDRK}.
The following result shows different convergence rates along different singular vectors of $ A $.

\begin{theorem}
	\label{main-ThmrRDRK}
	Suppose that $x^*$ is a solution of the linear system $Ax=b$ and $\alpha\in(0,1)$.
	For any $v_i$ being the (right) singular vector of $A$ associated to the singular
	value $\sigma_{i}(A)$, the iteration sequence $\{x^k\}_{k=0}^{\infty}$ generated by Algorithm \ref{r-RDRK} satisfies
	$$
	\mathbb{E} \left[ \left\langle x^k-x^*,v_{i}\right\rangle \right] =\left((1-\alpha)+\alpha\left(1-\frac{2\sigma^2_{i}(A)}{\|A\|^2_F}\right)^r\right)^k\left\langle x^0-x^*,v_{i}\right\rangle.
	$$
\end{theorem}

\begin{remark}
	Unlike the cases in Theorem \ref{main-ThmrRDR} and \ref{main-ThmrRDRv2}, $x^*$ in Theorem \ref{main-ThmrRDRK} is an arbitrary solution of the linear system. Certainly, one can take  $x^*_0$.  We also note that such convergence is weaker than that of $\mathbb{E}[ \|x^k-x^*\|^2_2 ]$.
	Consider the example given in Section \ref{sect:DR} where $x^k=x^*+r^k$ with $r^k$ being drawn i.i.d. from $N(0, I/\sqrt{n})$.
	Such a sequence satisfies $\mathbb{E} \left[ \left\langle x^k-x^*,v_{i}\right\rangle \right]=0$, while $\mathbb{E} [ \|x^k-x^*\|^2_2 ] \equiv 1$.
\end{remark}

Theorem \ref{main-ThmrRDRK} exhibits that the RrDR method converges exponentially along different singular directions of $ A $ at different rates depending on the singular values.
It accounts for the typical semiconvergence phenomenon.
That is, the residual $ \|Ax^{k}-b\|_2^2 $ decays faster at the beginning, but then gradually stagnates.
Recently, the semiconvergence phenomenon has been exploited by Wu and Xiang \cite{Wu21} for the randomized row iterative method \cite{Gow15}. They generalized the study in \cite{Jia17} and split the total error into the low- and high-frequency solution spaces.
In the literature, several acceleration techniques have been proposed to avoid semiconvergence phenomenon, for instance, the weighted version \cite{Ste20Wei} and momentum acceleration technique \cite{loizou2020momentum}. %are notable examples of such approaches.
In this paper, we will introduce the momentum acceleration technique to the RrDR method.

\section{Momentum acceleration}\label{sec:mrrdr}
\label{section-mRrDR}

In this section, we provide the momentum induced RrDR (mRrDR) method for solving feasibility problem derived from  linear systems. First, we give a short description of the heavy ball momentum method. Consider the  unconstrained minimization problem
$
\min\limits_{x\in\mathbb{R}^n} f(x),
$
where $f$ is a differentiable convex function.
To solve the problem, the gradient descent method
with momentum (HBM) of Polyak \cite{polyak1964some} takes the form
\begin{equation}\label{hbm}
	x^{k+1}=x^{k}-\alpha \nabla f\big(x^{k}\big)+\beta\big(x^{k}-x^{k-1}\big),
\end{equation}
where $\alpha>0$ is the stepsize, $\beta$ is the momentum parameter, and $\nabla f\left(x^{k}\right)$ denotes the gradient of $f$ at $x^k$.  When $\beta=0$, the method reduces to the gradient descent method. If the full gradient in \eqref{hbm} is replaced by the unbiased estimate of the gradient, then it becomes the  stochastic
HBM (SHBM) method.
In \cite{ghadimi2015global}, the authors showed that the deterministic HBM method
converges globally and sublinearly for smooth and convex functions. For the SHBM, one may refer to \cite{sebbouh2021almost,garrigos2023handbook} for more discussions.

Inspired by the success of the SHBM method, in this section  we incorporate the HBM into our RrDR method, obtaining the mRrDR method described in Algorithm \ref{r-mRDRK}. To the best of our knowledge, this is the first time that momentum variants of the $r$-sets-DR method are investigated. In the rest of this section, we will study the convergence properties of the proposed mRrDR method.

\begin{algorithm}[htpb]
	\caption{Randomized $r$-sets-Douglas-Rachford with momentum (mRrDR) \label{r-mRDRK}}
	\begin{algorithmic}
		\Require
		$A\in \mathbb{R}^{m\times n}$, $b\in \mathbb{R}^m$, $r\in\mathbb{Z}_{+}$, $k=1$,  extrapolation/relaxation parameter $\alpha$, the heavy ball momentum parameter $\beta$, and initial vectors $x^1,x^0\in \mathbb{R}^{n}$.
		\begin{enumerate}
			\item[1:] Set $z^{k}_0:=x^k$.
			\item[2:] {\bf for $\ell=1,\ldots,r$ do}
			\item[3:] \ \ \ Select $j_{k_{\ell}}\in\{1,\ldots,m\}$ with probability $\mbox{Pr}(\mbox{row}=j_{k_{\ell}})=\frac{\|a_{j_{k_{\ell}}}\|^2_2}{\|A\|_{F}^2}$.
			\item[4:] \ \ \ Compute
			$$
			z_{\ell}^{k}:=z_{\ell-1}^k-2\frac{\langle a_{j_{k_{\ell}}},z_{\ell-1}^k\rangle-b_{j_{k_{\ell}}}}{\|a_{j_{k_\ell}}\|^2_2}a_{j_{k_{\ell}}}.
			$$
			\item[5:] {\bf end for}
			\item[6:] Update
			$$
			x^{k+1}:=(1-\alpha) x^k+\alpha z_{r}^k+\beta(x^k-x^{k-1}).
			$$
			\item[7:] If the stopping rule is satisfied, stop and go to output. Otherwise, set $k=k+1$ and return to Step $1$.
		\end{enumerate}
		
		\Ensure
		The approximate solution $ x^k $.
	\end{algorithmic}
\end{algorithm}

\subsection{Convergence of iterates}
We have the following convergence result for Algorithm \ref{r-mRDRK}.
\begin{theorem}\label{THMfmm}
	Suppose that  the linear system $Ax=b$ is consistent,
	$x^1=x^0\in\mathbb{R}^n$ are arbitrary initial vectors, and $x_0^*=A^{\dagger}b+(I-A^\dagger A)x^0$.	
	Assume $0<\alpha<1$, $\beta\geq0$ and that
	$$
	\gamma_1:=\alpha^2+(1-\alpha)^2+\left(2\alpha(1-\alpha)+3\alpha\beta\right)
	\left(1-2\frac{\sigma_{\min}^2(A)}{\|A\|^2_F}\right)^r
	+2\beta^2+3(1-\alpha)\beta
	$$
	and
	$$
	\gamma_2:=2\beta^2+(1-\alpha)\beta+\alpha\beta\left(1-2\frac{\sigma_{\min}^2(A)}{\|A\|^2_F}\right)^r
	$$
	satisfy $\gamma_1+\gamma_2<1$.
	Then the iteration sequence  $\{x^k\}_{k=0}^{\infty}$ generated by Algorithm \ref{r-mRDRK} satisfies
	$$
	\mathbb{E} [ \|x^{k+1}-x_{0}^*\|^2_2 ] \leq q^k(1+\tau)\|x^0-x_{0}^*\|^2_2,\ \forall \ k\geq 0,
	$$
	where $q=\frac{\gamma_1+\sqrt{\gamma_1^2+4\gamma_2}}{2}$ and $\tau=q-\gamma_1\geq 0$. Moreover,
	$\gamma_1+\gamma_2\leq q<1$.
\end{theorem}

We here provide one approach to choose the parameters $\alpha$ and $\beta$. Specifically, letting
$$
\tau_1:=4(1-\alpha)+4\alpha\left(1-2\frac{\sigma_{\min}^2(A)}{\|A\|^2_F}\right)^r \ \text{and} \
\tau_2:=2\alpha(1-\alpha)\left(1-\left(1-2\frac{\sigma_{\min}^2(A)}{\|A\|^2_F}\right)^r\right),
$$
if the parameters $\alpha$ and $\beta$ are chosen such that $$
0<\alpha<1 \ \text{and} \ 0\leq \beta<\frac{1}{8}\left(\sqrt{\tau_1^2+16\tau_2}-\tau_1\right),
$$
then it can be verified that  $ \gamma_1+\gamma_2<1$.

Next, we compare the convergence rates  obtained in Theorems \ref{main-ThmrRDR} and \ref{THMfmm}. From the definition of $\gamma_1$ and $\gamma_2$, we know that the convergence rate $q(\beta)$  in Theorem \ref{THMfmm} can be viewed as a function of $\beta$. Since $\beta\geq0$, we have $\tau_2\geq0$. As $\gamma_1+\gamma_2<1$, it implies that
$\gamma_1\gamma_2 + \gamma_2^2 = \gamma_2(\gamma_1+\gamma_2) \leq\gamma_2$. Therefore,
$\gamma_1^2+4\gamma_2 \geq(\gamma_1+2\gamma_2)^2$
and hence,
$$
\begin{aligned}
	q(\beta)&\geq\gamma_1+\gamma_2=4\beta^2+\tau_1\beta-\tau_2+1\geq 1-\tau_2
	\\
	&=q(0)=\alpha^2+(1-\alpha)^2+2\alpha(1-\alpha)\left(1-2\frac{\sigma_{\min}^2(A)}{\|A\|^2_F}\right)^r.
\end{aligned}
$$
Clearly, the lower bound on $q$ is an increasing function of $\beta$, which implies that for any $\beta$ the rate is always inferior  to that of Algorithm \ref{r-RDRK} in Theorem \ref{main-ThmrRDR}.

\subsection{Accelerated linear rate of expected iterates}
%To compensate for the negative conclusion in the previous subsection, we now show that by a proper choice of
To theoretically exhibit the enhancement of the heavy ball momentum, we now show that with a proper selection of
the relaxation parameter $\alpha$ and momentum parameter $\beta$, Algorithm \ref{r-mRDRK} enjoys an accelerated linear convergence rate  in terms of $\| \mathbb{E}  [  x^k-x_0^* ] \|_2^2$.

\begin{theorem}\label{THMfm2}
	Suppose that the linear system $Ax=b$ is consistent,
	$x^1=x^0\in\mathbb{R}^n$ are arbitrary initial vectors, and $x_0^*=A^{\dagger}b+(I-A^\dagger A)x^0$.
	Let $\{x^k\}_{k=0}^{\infty}$ be the iteration sequence in Algorithm \ref{r-mRDRK}.
	Assume that the relaxation parameter
	$$
	0<\alpha<\min\left\{1, \frac{1}{1-\left(1-2\sigma_{\max}^2(A)/\|A\|^2_F\right)^{r}}\right\}$$
	and  the momentum parameter
	$$ \left(1-\sqrt{\alpha\left(1-\left(1-2\frac{\sigma_{\min}^2(A)}{\|A\|^2_F}\right)^r\right)}\right)^2<\beta<1.$$
	Then there exists a constant
	$c>0$ such that
	$$
	\| \mathbb{E}[  x^k-x_0^*] \|_2^2\leq \beta^k c.
	$$
\end{theorem}

\begin{remark}
	Note that the convergence factor in Theorem  \ref{THMfm2} is  equal to the value of the momentum parameter $\beta$.
	Theorem \ref{main-ThmrRDRv2} shows that Algorithm \ref{r-RDRK} (without momentum)
	converges with iteration complexity
	$$
	O\left( \log(\varepsilon^{-1})\left(\alpha\left(1-\left(1-2\sigma_{\min}^2(A)/\|A\|^2_F\right)^r\right)\right)^{-1}\right).
	$$
	In contrast, based on Theorem \ref{THMfm2} we have,
	for $\beta=\left(1-\sqrt{0.99\alpha\left(1-\left(1-2\sigma_{\min}^2(A)/\|A\|^2_F\right)^r\right)}\right)^2$, the iteration complexity of Algorithm \ref{r-mRDRK} is
	$$
	O\left(\log (\varepsilon^{-1})\sqrt{\left(0.99\alpha\left(1-\left(1-2\sigma_{\min}^2(A)/\|A\|^2_F\right)^r\right)\right)^{-1}} \right),
	$$
	which is a quadratic improvement on the above conclusion.
	%To sum up, mRrDR method converges linearly in terms of $ 	\mathbb{E} [ \|x^k-x^*_0 \|^2_2] $ yet the rate is not satisfying. while Concerning a weak convergence $ \mathbb{E} \left[ \|x-\mathbb{E}[x]\|_2^{2}\right]$ the acceleration is presented. The next section will test the actual performance of mRrDR.
\end{remark}

\section{Numerical experiments}\label{sec:ne}

In this section, we study the computational behavior of the two proposed algorithms,
RrDR and mRrDR. In particular, we focus mainly on the evaluation of the performance
of mRrDR. We compare  mRrDR with some of the state-of-the-art methods, namely, RK \cite{Str09}, RGS \cite{Lev10,Gri12}, and RP-ADMM \cite{Sun20}. Moreover, we also compare the implementation of mRrDR  with the  built-in function \texttt{pinv} and \texttt{lsqminnorm} in {\sc Matlab}.

All the methods are implemented in  {\sc Matlab} R2019b for Windows $10$ on a desktop PC with the  Intel(R) Core(TM) i7-10710U CPU @ 1.10GHz  and 16 GB memory.

\subsection{Numerical setup}
We mainly use the following three types of data for our test.

{\bf Synthetic data.}
For synthetic data, given different values of $\|A\|^2_F/\sigma^2_{\min}(A)$, we generate a group of matrices $A$. We then generate the exact solution $x^*$ by $ x^*=(A^\top w)/\|A^\top w\|_2 $ with a random $ w \in \mathbb{R}^m $ to ensure that it lies in the $ \operatorname{Row}(A) $ and $b=Ax^*$ to ensure the consistency of the system. The synthetic data are designed to investigate the influence of the rate coefficient on the convergence process.

{\bf Real-world data.} The real-world data are available via the SuiteSparse Matrix Collection \cite{Kol19} and LIBSVM \cite{chang2011libsvm}.
In our experiments, we only use the matrices $A$ of the datasets. If $m<n$, then we use $A^\top$ as the coefficient matrix. Similarly, to ensure the consistency of the linear system, we first generate the solution by $ x^*=(A^\top w)/\|A^\top w\|_2 $ and then set $b=Ax^*$.

{\bf Average consensus.}
Suppose $G=(V, E)$ is an undirected connected network with the vertex set $V=\{v_1, v_2, \cdots, v_n\}$ and the edge set $E \ (|E|=m)$.
In the average consensus (AC) problem, each vertex $v_i \in V$ owns a private value $c_{i} \in \mathbb{R}$,
and the goal of the problem is to compute the average of the private values of each vertex of the network,
$\bar{c}:=\frac{1}{n} \sum_{i} c_{i}$, where only communication between neighbours is allowed.
The problem is fundamental in distributed computing and multiagent systems \cite{boyd2006randomized,loizou2021revisiting}, and has many applications such as PageRank, coordination of autonomous agents, and rumor spreading in social networks. Recently, under an appropriate setting, the famous randomized pairwise gossip algorithm \cite{boyd2006randomized} for solving the AC problem has been proved to be equivalent to the RK method. One may refer to \cite{loizou2021revisiting} for more details.

For the synthetic data and real data,  the initial vector is chosen as $x^0=0$ (or $x^1=x^0=0$ for mRrDR). For the AC problem, the initial vector is chosen as $x^0=c$ (or $x^1=x^0=c$ for mRrDR).
The computations are terminated once the relative solution error (RSE), defined as
$\operatorname{RSE}= \|x^k-x_0^*\|^2_2 / \|x^0-x_0^*\|^2_2$,
is less than a specific error tolerance or the number of iteration exceeds a certain limit.  All the results below are averaged over $10$ trials.

\subsection{Optimal selection of the parameters}
As observed in the convergence theorems presented in sections \ref{section-RrDR} and  \ref{section-mRrDR}, the parameters of RrDR and mRrDR have an influence on the convergence rate.
%there are parameters that have an influence on the convergence rate of the method:
%the number of the chosen sets $r$, the relaxation parameter $\alpha$, and the momentum parameter $\beta$ at each iteration of RrDR or mRrDR. 
In this subsection, we aim to find relatively favorable choices via numerical tests. All computations are terminated once $\operatorname{RSE}<10^{-12}$.
%we aim to find the parameters that optimize the convergence rate in terms of total computation.
%an important role in the convergence rate of the method.
\subsubsection{Choice of $\alpha$ and $\beta$}
In this subsection, we demonstrate the computational behavior of mRrDR with respect to different parameters $\alpha$ and $\beta$.
(Note that for $\beta=0$ it is equivalent to RrDR.) 
%Note that for $\beta=0$ the method is equivalent to its no-momentum variant RrDR.) 
We measure the performance of the method with respect to the number of iterations. From Figures \ref{figue620}, \ref{figue620-1}, \ref{figue620-22}, and \ref{figue628-1}, it is obvious that the introduction of momentum term leads to an improvement in the performance of RrDR.
% and that $(\alpha,\beta)=(0.5,0.4)$ is a good option for a satisfactory convergence of mRrDR. %The details of the numerical experiments are as follows.
More specifically, from the figures we observe the following.
%For each linear system, we run mRrDR for several different relaxation parameters $\alpha$ and
%momentum parameters $\beta$. We measure the performance of the
%method with respect to the number of iterations (average of 10 trials). Note that for $\beta=0$ the method is equivalent to its no-momentum variant RrDR.
%From Figures \ref{figue620}, \ref{figue620-1}, \ref{figue620-22} and \ref{figue628-1}, it is clear that the introducing of momentum term leads to an improvement in the performance of RrDR. More specifically, from the figures we observe the following:

\begin{itemize}
	\item[1.] The momentum technique can improve the convergence behavior of the method. %It can be observed that f
	For any fixed value of $\alpha$, with appropriate momentum parameters $0 <\beta \leq 0.9$, mRrDR typically converges faster than its non-momentum variant RrDR.
	\item[2.] For the RrDR method ($ \beta = 0 $), a larger value of $\alpha$ can help achieve a faster convergence. This is consistent with the observation in the literature that the overrelaxation parameter is more advisable for better performance \cite{Str09,Nec19}.
	\item[3.] For different values of $\|A\|^2_F/\sigma^2_{\min}(A)$ (well- or ill-conditioned linear systems),  and different choices of $r$, $(\alpha,\beta)=(0.5,0.4)$ is typically a good option for a sufficient fast  convergence of mRrDR.
\end{itemize}

\begin{figure}[hptb]
	\centering
	\begin{tabular}{cc}
		\includegraphics[width=0.31\linewidth]{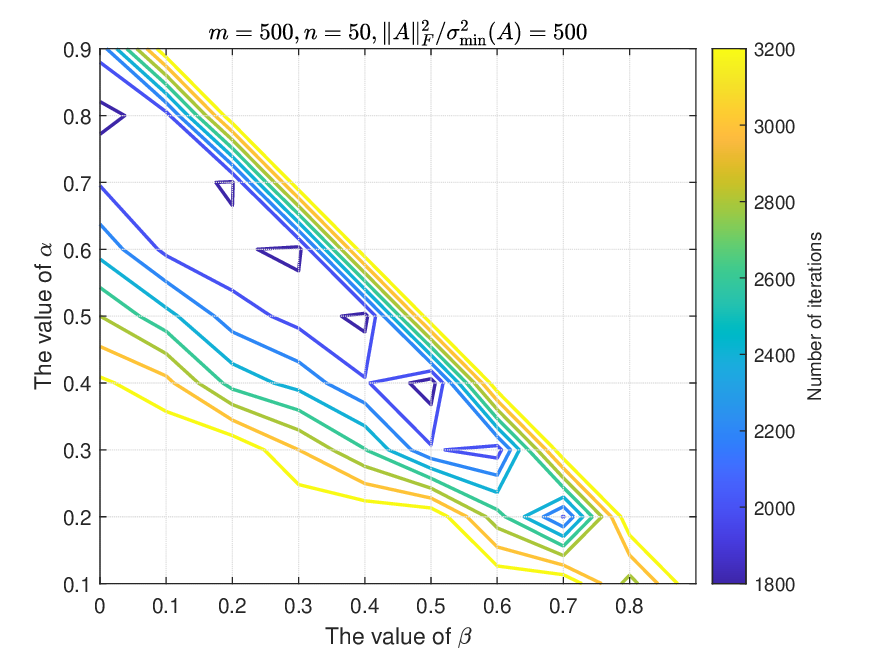}
		\includegraphics[width=0.31\linewidth]{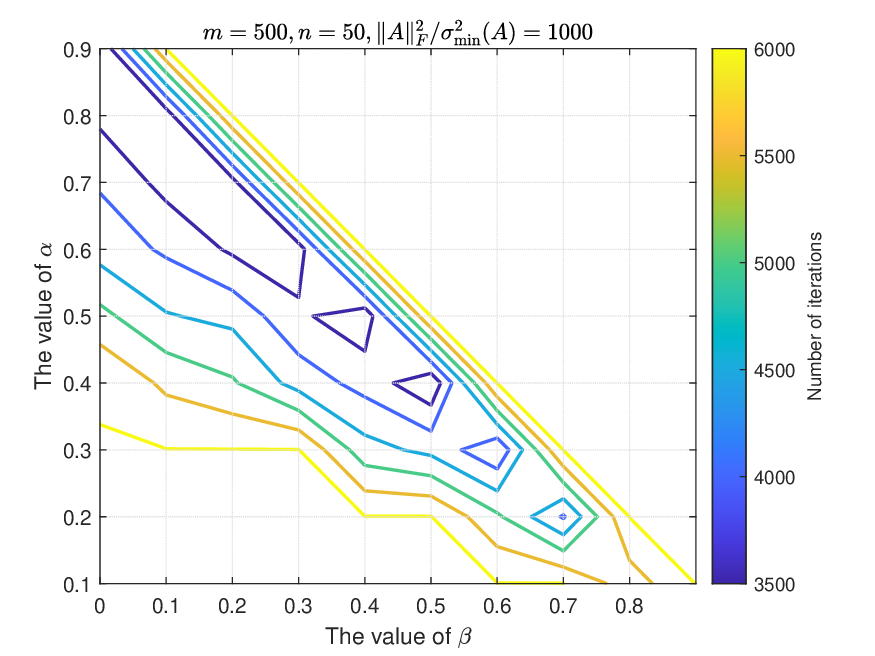}
		\includegraphics[width=0.31\linewidth]{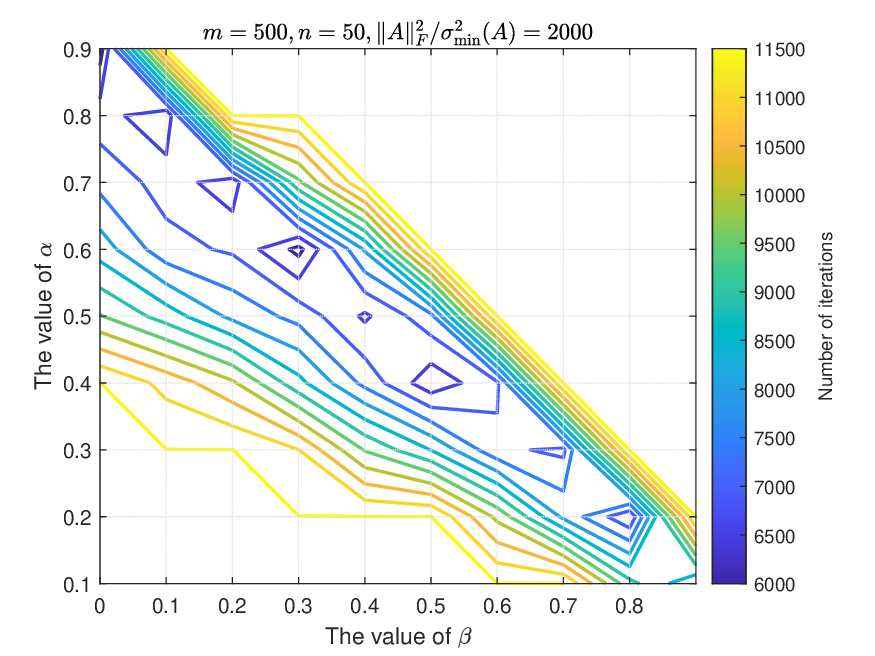}\\
		\includegraphics[width=0.31\linewidth]{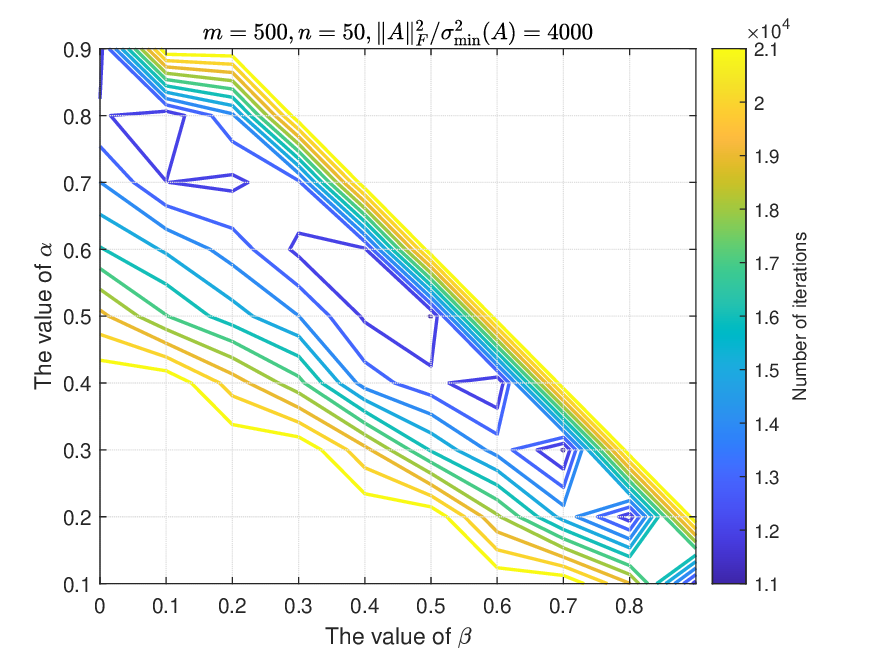}
		\includegraphics[width=0.31\linewidth]{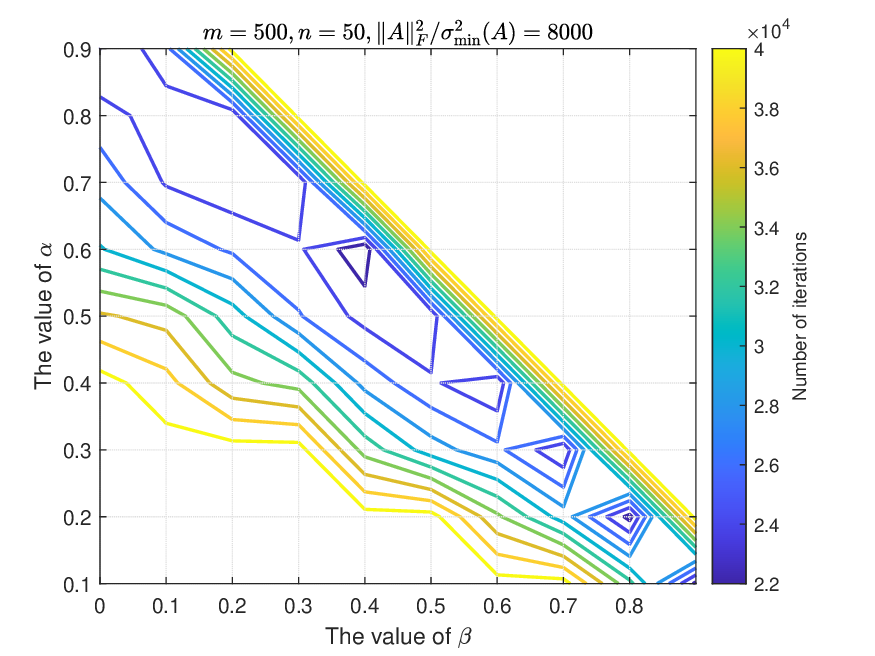}
		\includegraphics[width=0.31\linewidth]{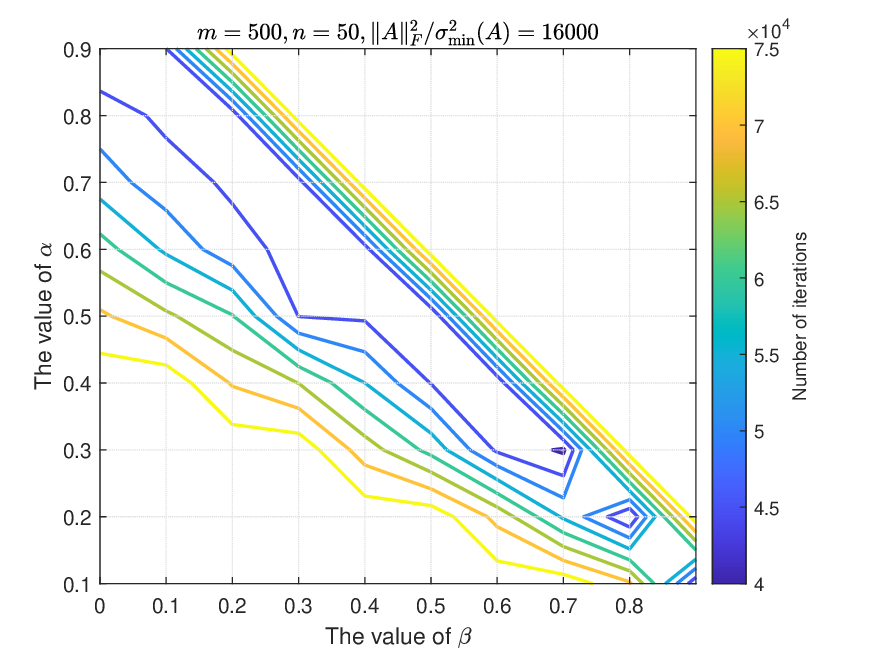}
	\end{tabular}
	\caption{Performance of mRrDR with different parameters $\alpha$ and $\beta$ for consistent linear systems with Gaussian matrix $A$, where $r=2$. The title of each plot indicates the dimensions of the matrix $A$ and the value of $\|A\|^2_F/\sigma_{\min}^2(A)$. }
	\label{figue620}
\end{figure}

\begin{figure}[hptb]
	\centering
	\begin{tabular}{cc}
		\includegraphics[width=0.31\linewidth]{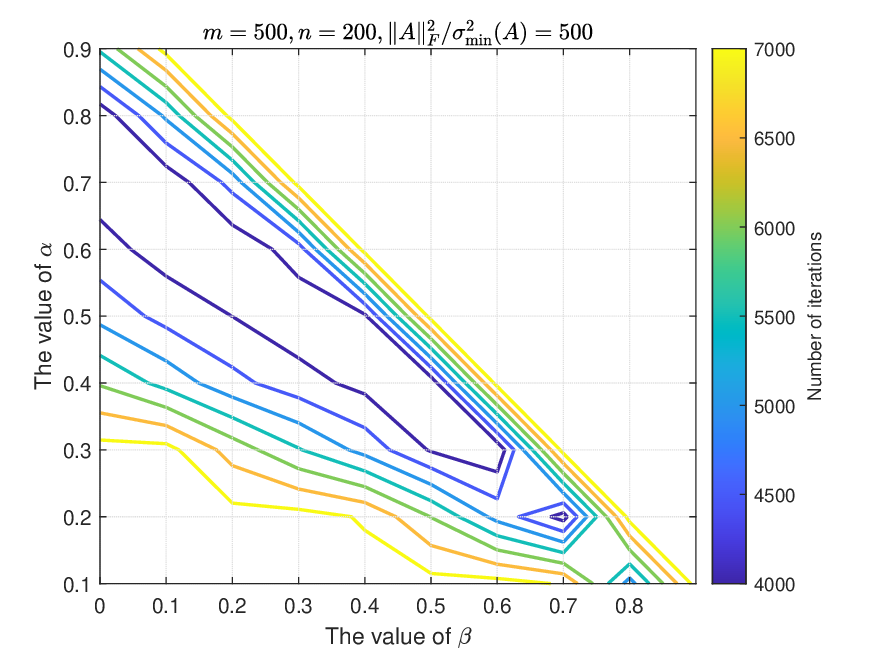}
		\includegraphics[width=0.31\linewidth]{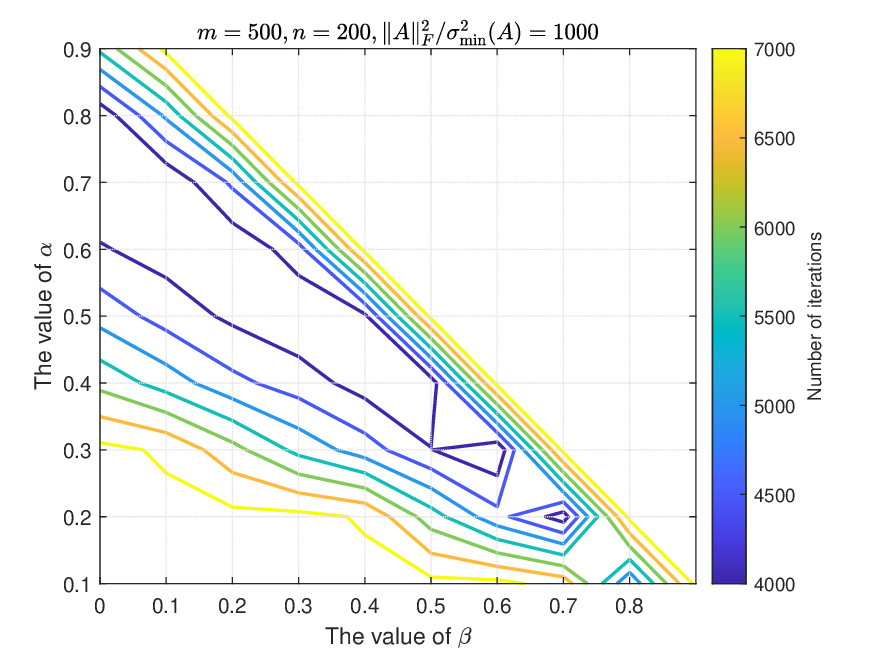}
		\includegraphics[width=0.31\linewidth]{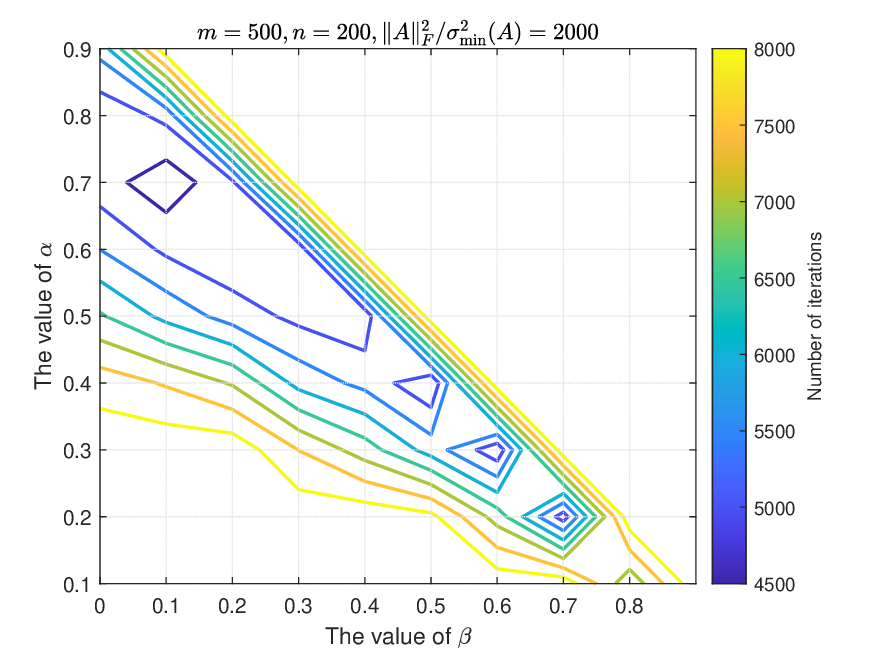}\\
		\includegraphics[width=0.31\linewidth]{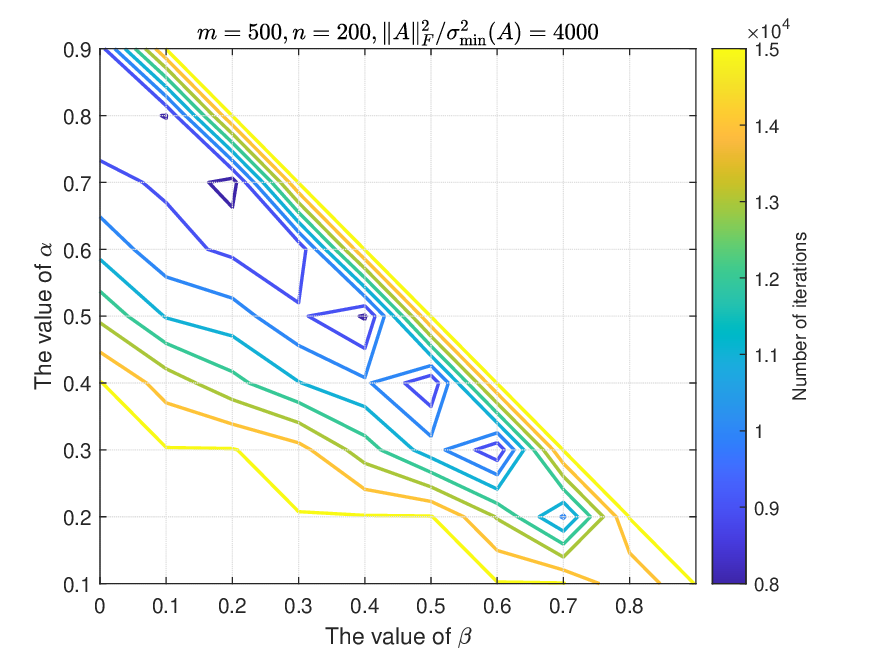}
		\includegraphics[width=0.31\linewidth]{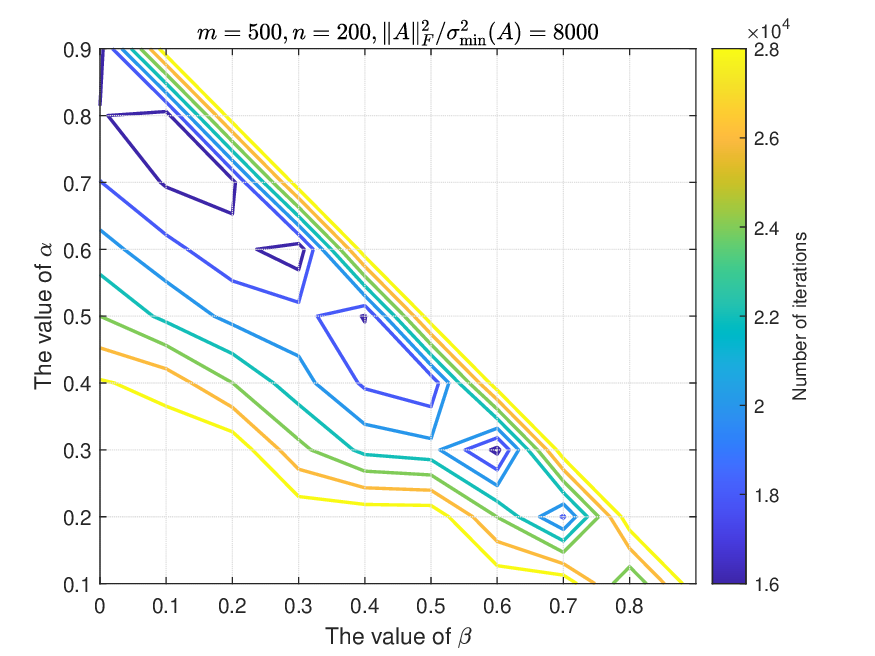}
		\includegraphics[width=0.31\linewidth]{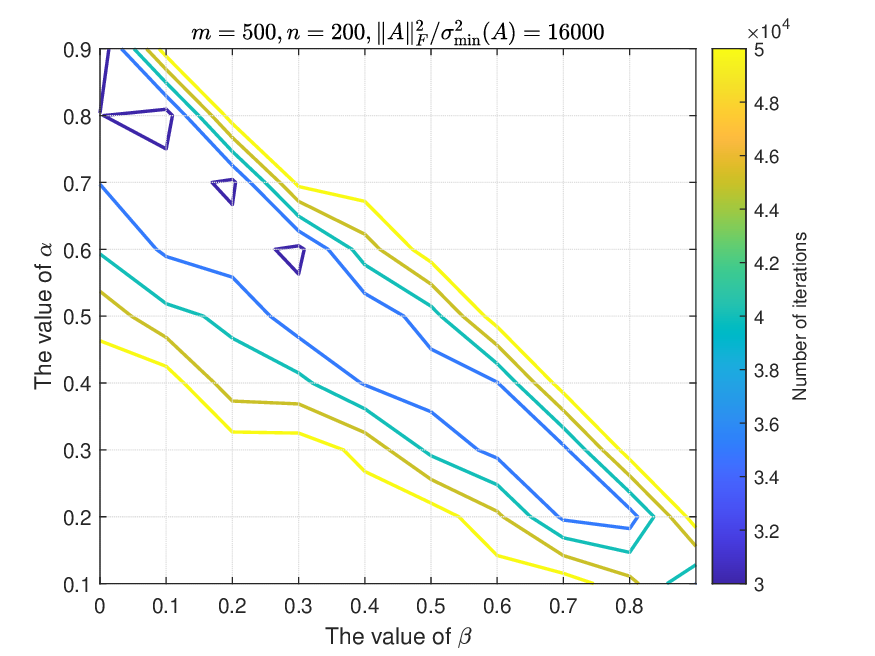}
	\end{tabular}
	\caption{Performance of mRrDR with different parameters $\alpha$ and $\beta$ for consistent linear systems with Gaussian matrix $A$, where $r=3$. The title of each plot indicates the dimensions of the matrix $A$ and the value of $\|A\|^2_F/\sigma_{\min}^2(A)$. }
	\label{figue620-1}
\end{figure}

\begin{figure}[hptb]
	\centering
	\begin{tabular}{cc}
		\includegraphics[width=0.31\linewidth]{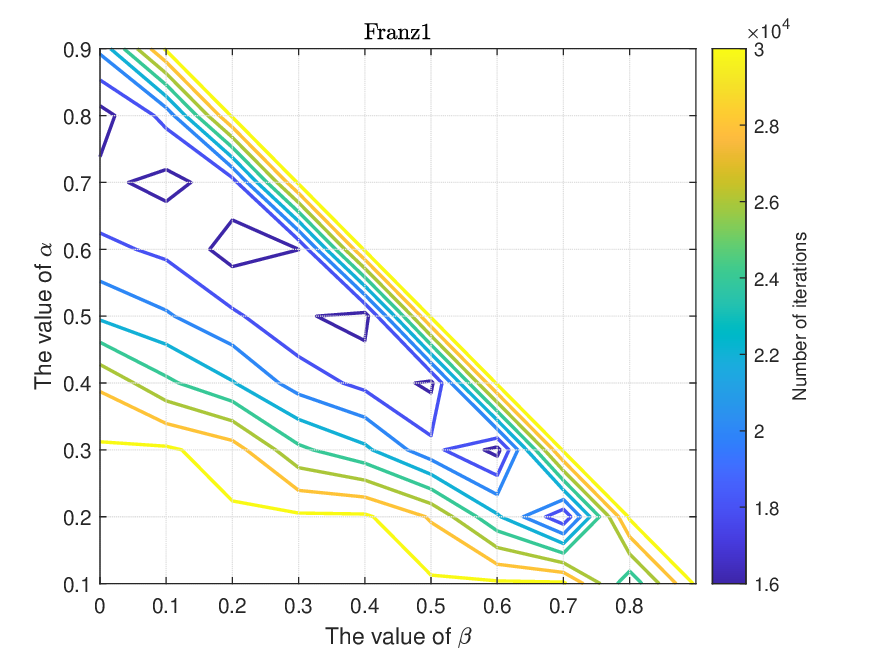}
		\includegraphics[width=0.31\linewidth]{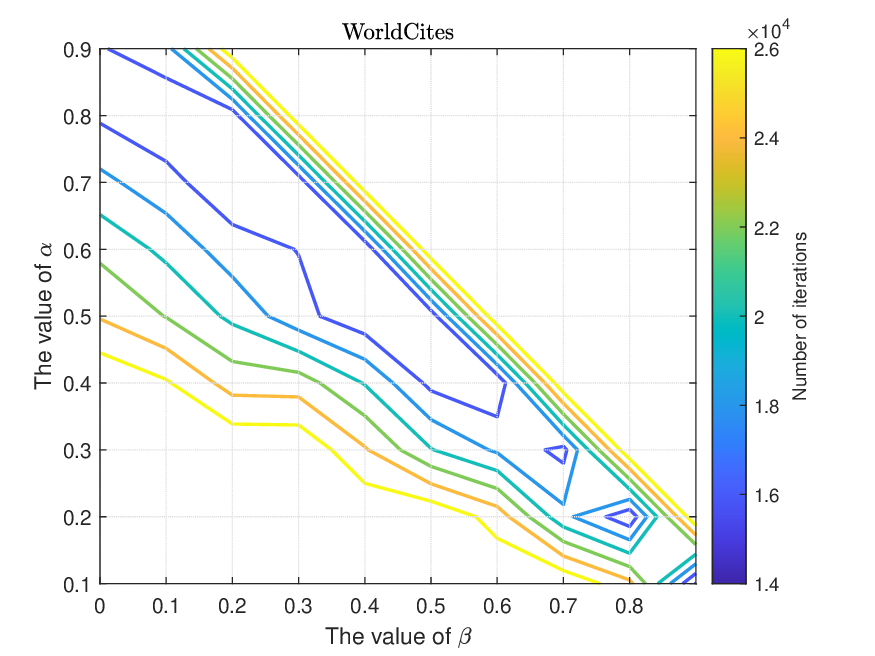}
		\includegraphics[width=0.31\linewidth]{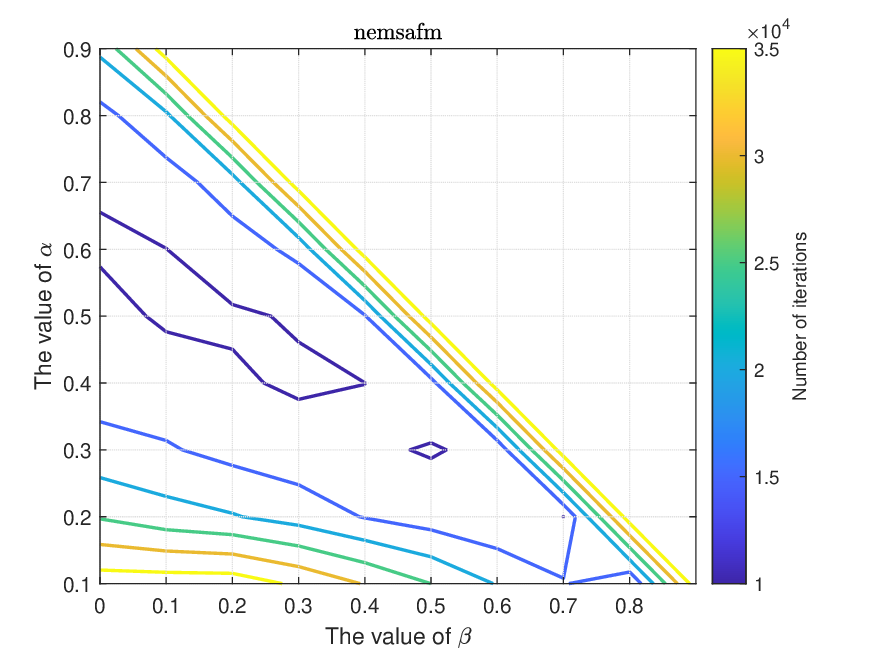}
	\end{tabular}
	\caption{Performance of mRrDR with different parameters $\alpha$ and $\beta$ on real data from  SuiteSparse Matrix Collection \cite{Kol19}. {\tt Franz1}: $(m,n)=(2240,768)$, {\tt WorldCites}: $(m,n)=(315,100)$, {\tt nemsafm}: $(m,n)=(2348,334)$, where $r=3$.}
	\label{figue620-22}
\end{figure}

\begin{figure}[hptb]
	\centering
	\begin{tabular}{cc}
		\includegraphics[width=0.31\linewidth]{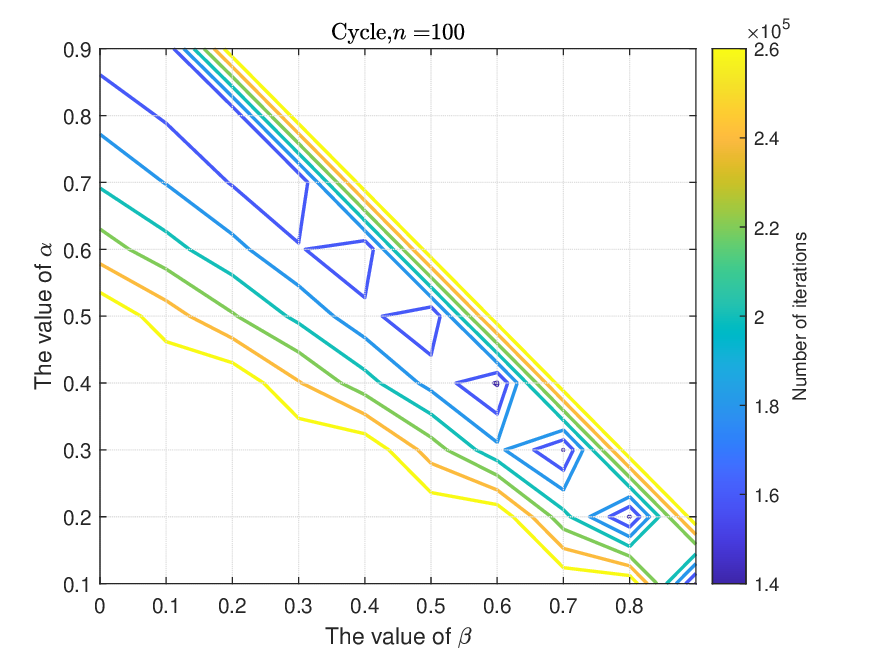}
		\includegraphics[width=0.31\linewidth]{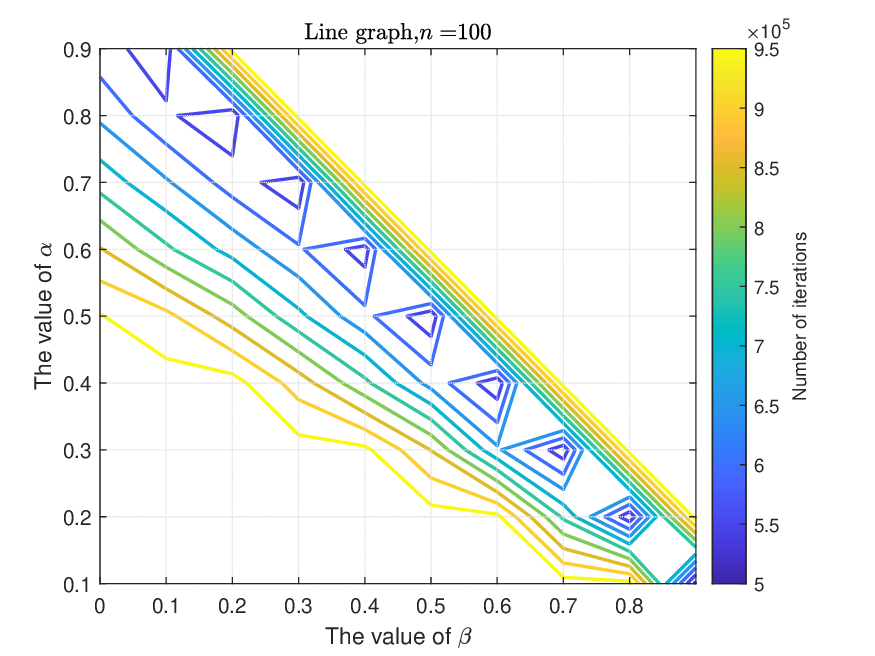}
		\includegraphics[width=0.31\linewidth]{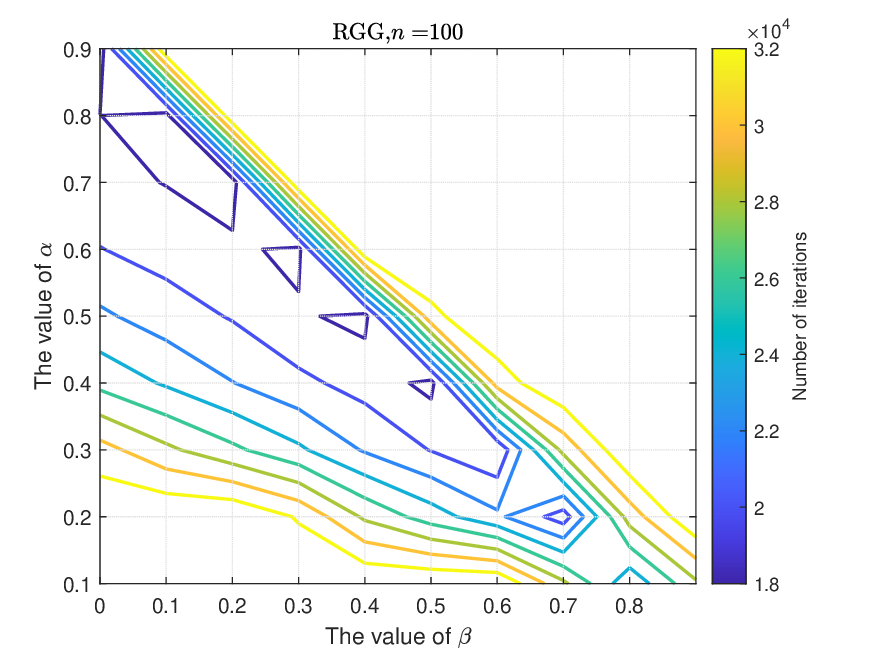}
	\end{tabular}
	\caption{Performance of mRrDR with different parameters $\alpha$ and $\beta$ on the AC problem, where $ r=2 $. }
	\label{figue628-1}
\end{figure}

\subsubsection{Choice of $r$}
\label{subsection:cr}
In this subsection, we investigate the computational performance of mRrDR with
respect to the parameter $r$.
We plot the performance of the method in terms of the number of row actions.
Figures \ref{figue6251} and \ref{figue6261} illustrate our experimental results. Note again that $\beta=0$ represents the RrDR method and $r=1$ represents the RK method.

It is can be seen that a larger $r$ may lead to slower convergence and $r\in[1, 10]$ is a good option. Besides, with an appropriate  $r$, the mRrDR method may converge faster than the mRK method $(r=1)$. For example, %it can be observed from
the first row in Figure \ref{figue6251} shows that mRrDR with $r=4$ performs better than the mRK method.%, i.e. $r=1$,% and from Figure \ref{figue6261} we know that $r=6$ or $r=7$ are good choices.

%%Our experiment results show that $1\leq r\leq 10$ is a good option for  better performances of mRrDR. %The details of the numerical experiments are as follows.
%
%For each linear system, we run mRrDR for several values of parameter $r$, and we plot the performance of the
%method (average of 10 trials) with respect to the number of row actions. Note again that $\beta=0$ represents the RrDR method and $r=1$ represents the RK method.
%
%Figures \ref{figue6251} and \ref{figue6261} illustrate our experimental results.  It is can be seen that a larger $r$ may lead to slower convergence and $1\leq r\leq 10$ is a good option. Besides, with an appropriate  $r$, the mRrDR method may converge faster than the mRK method $(r=1)$. For example, it can be observed from the first row in Figure \ref{figue6251} that mRrDR with $r=4$ performs better than the mRK method, i.e. $r=1$, and from Figure \ref{figue6261} we know that $r=6$ or $r=7$ are good choices.

\begin{figure}[hptb]
	\centering
	\begin{tabular}{cc}
		\includegraphics[width=0.31\linewidth]{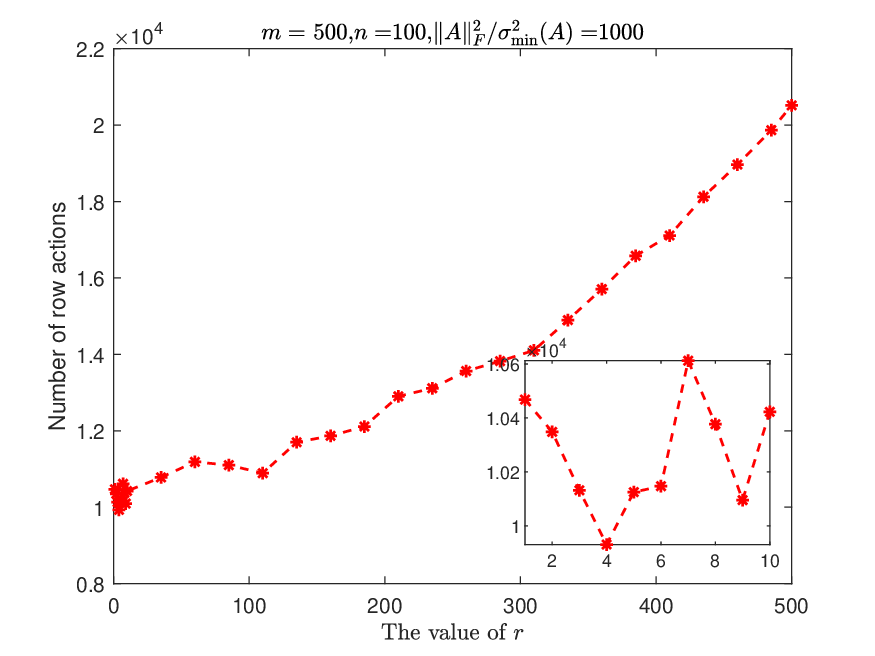}
		\includegraphics[width=0.31\linewidth]{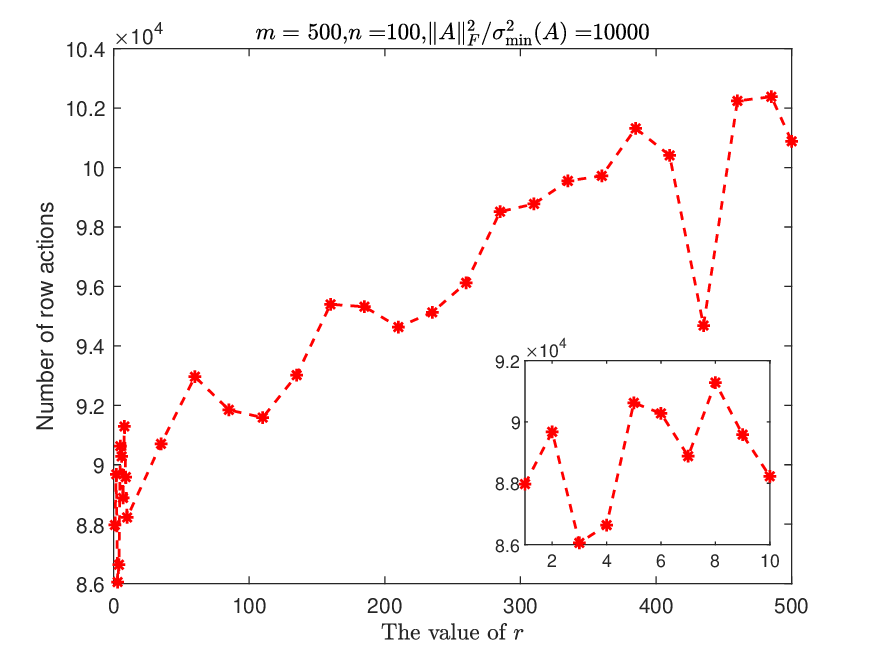}
		\includegraphics[width=0.31\linewidth]{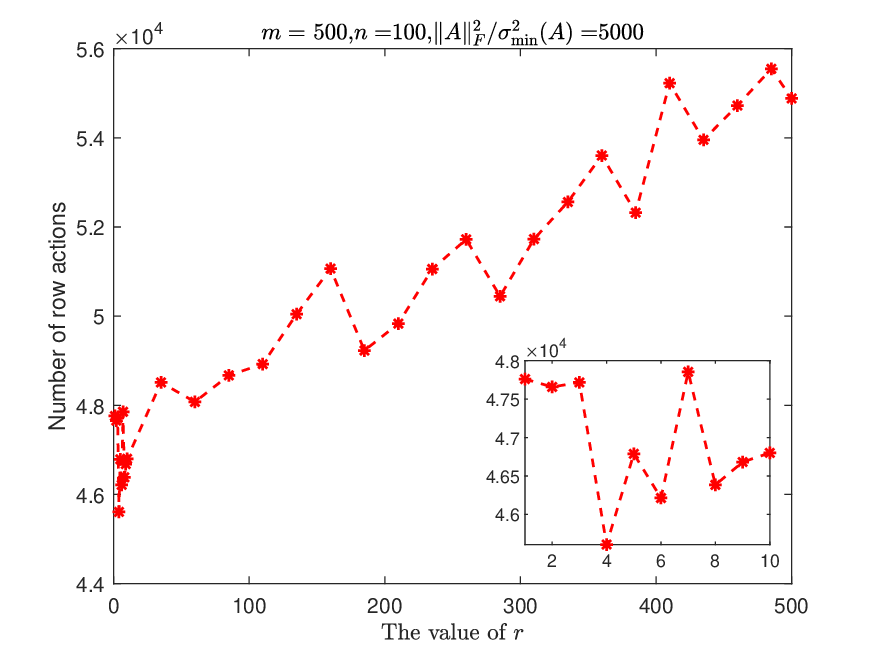}\\
		\includegraphics[width=0.31\linewidth]{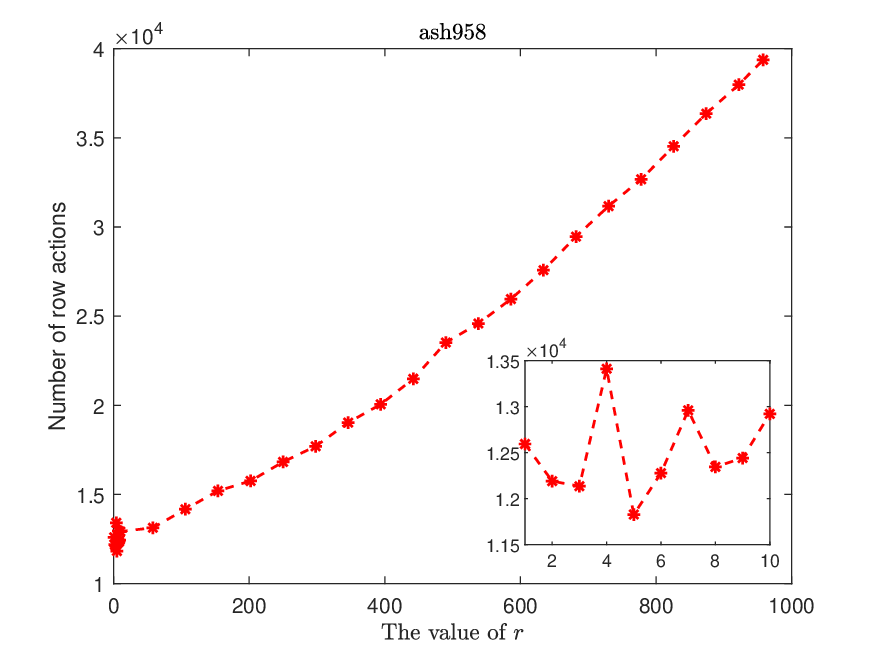}
		\includegraphics[width=0.31\linewidth]{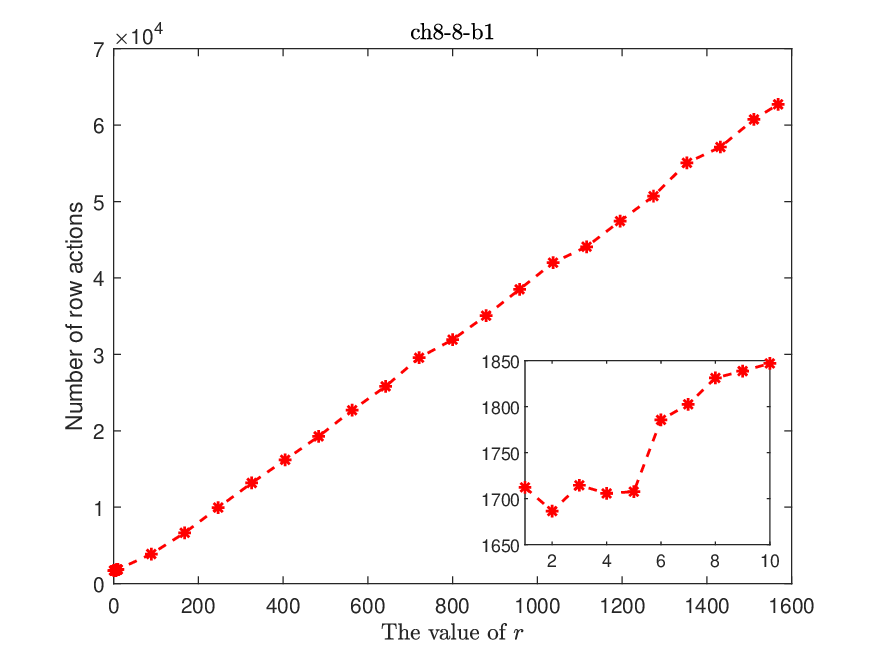}
		\includegraphics[width=0.31\linewidth]{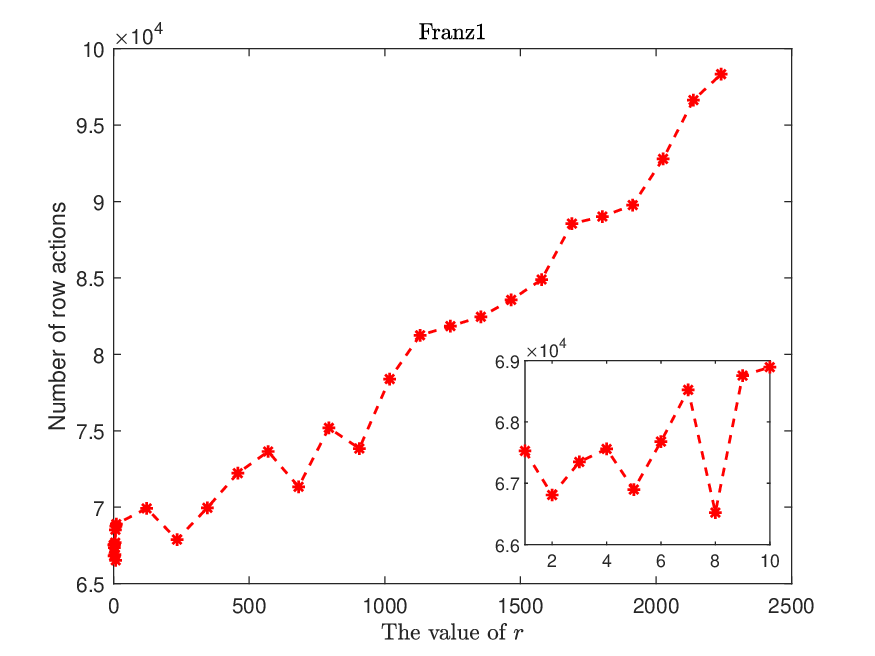}\\
		\includegraphics[width=0.31\linewidth]{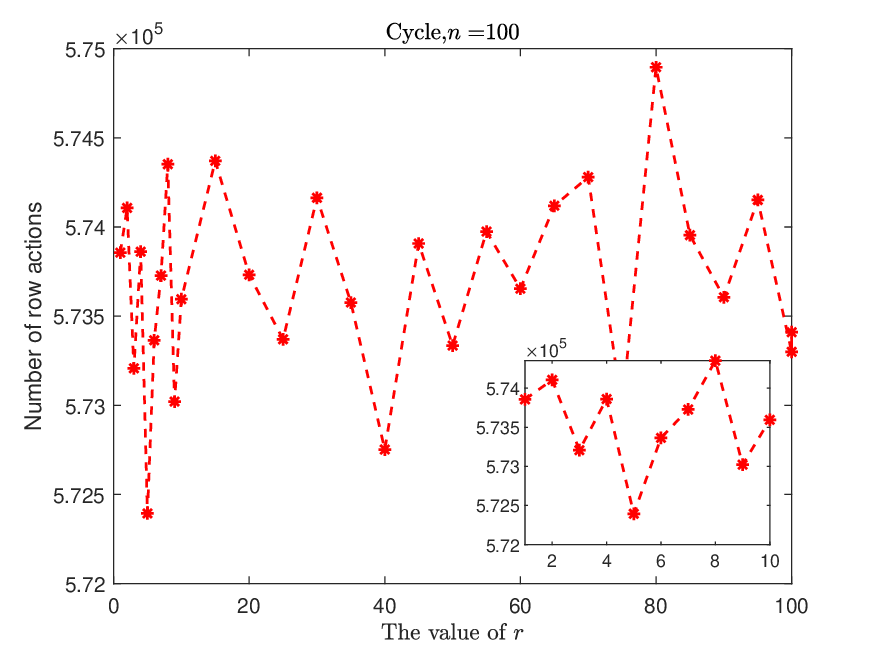}
		\includegraphics[width=0.31\linewidth]{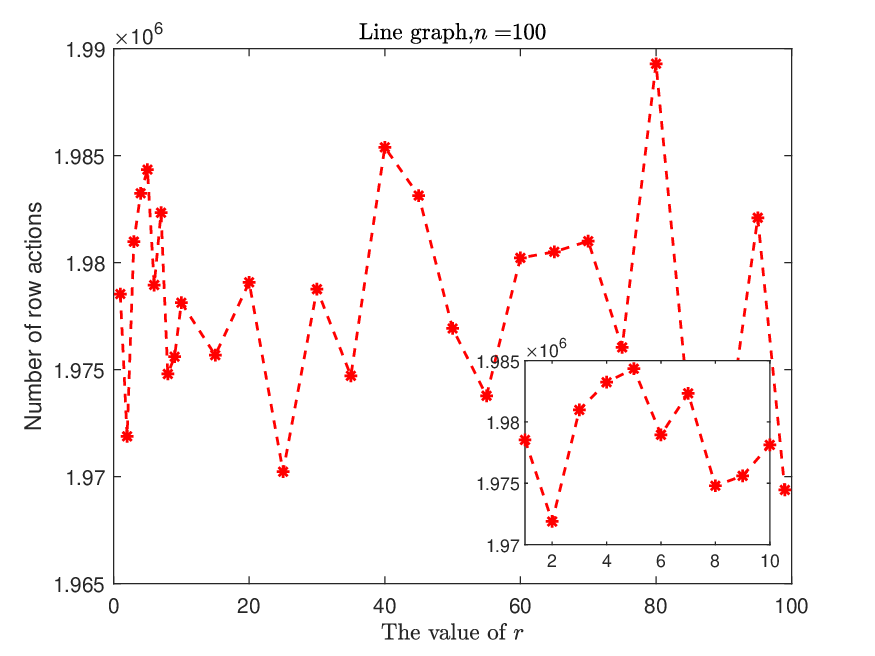}
		\includegraphics[width=0.31\linewidth]{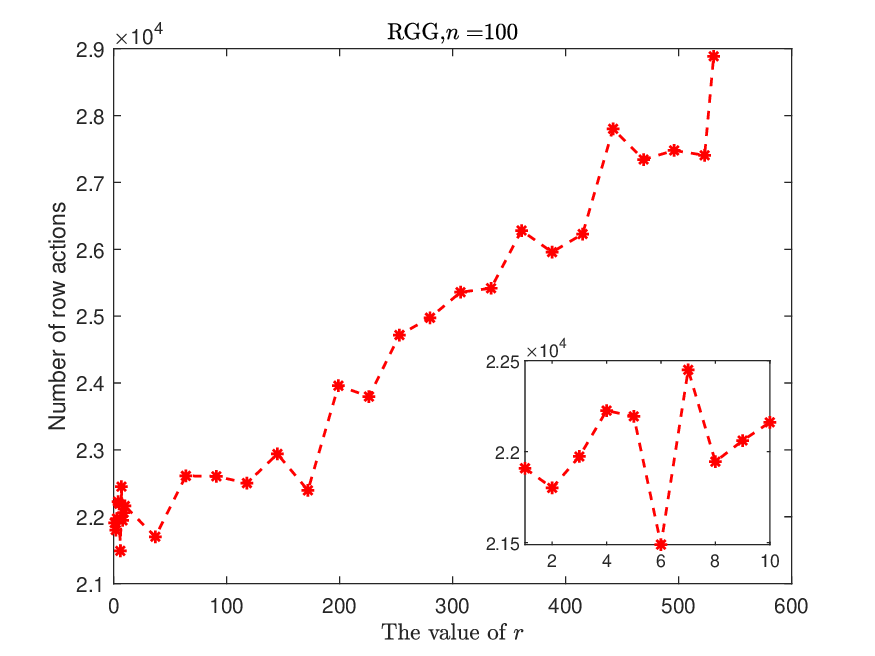}
	\end{tabular}
	\caption{Performance of mRrDR with different values of $r$. $r=1$ refers to the mRK method. The parameters $\alpha=0.5,\beta=0$.
		The title of each plot indicates the test data. The number of row actions is employed to illustrate the evolutions for the different settings of mRrDR. }
	\label{figue6251}
\end{figure}

\begin{figure}[hptb]
	\centering
	\begin{tabular}{cc}
		\includegraphics[width=0.31\linewidth]{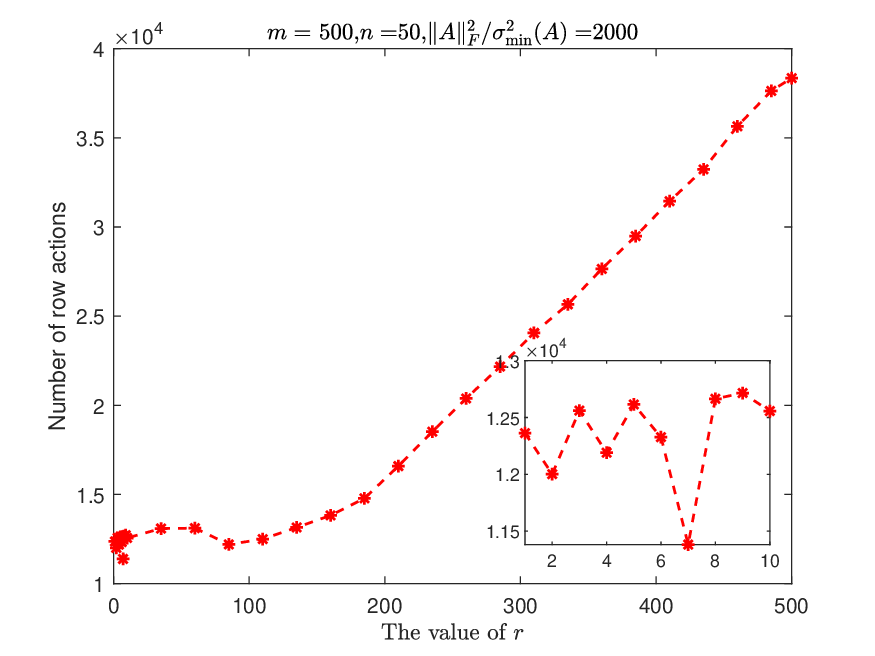}
		\includegraphics[width=0.31\linewidth]{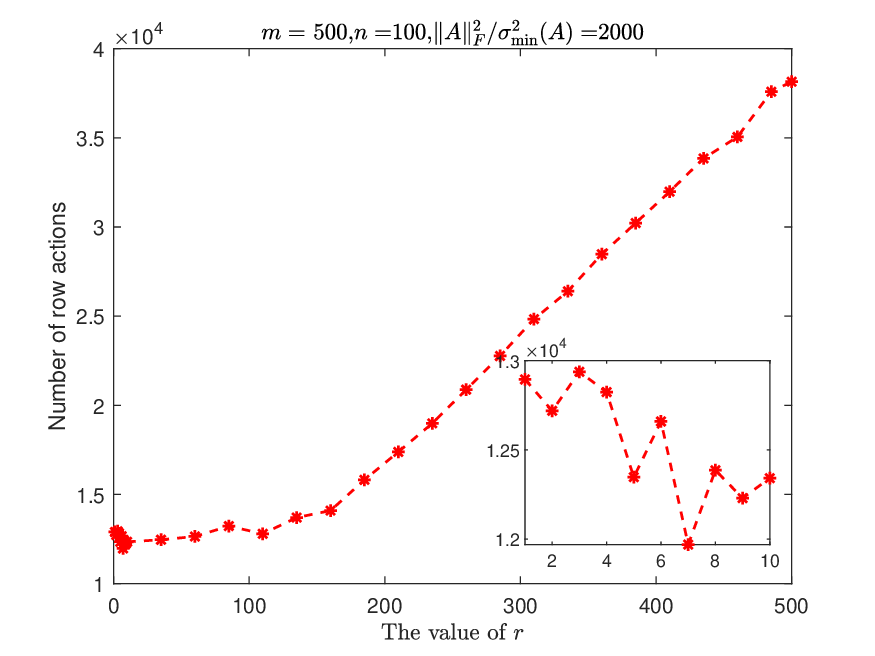}
		\includegraphics[width=0.31\linewidth]{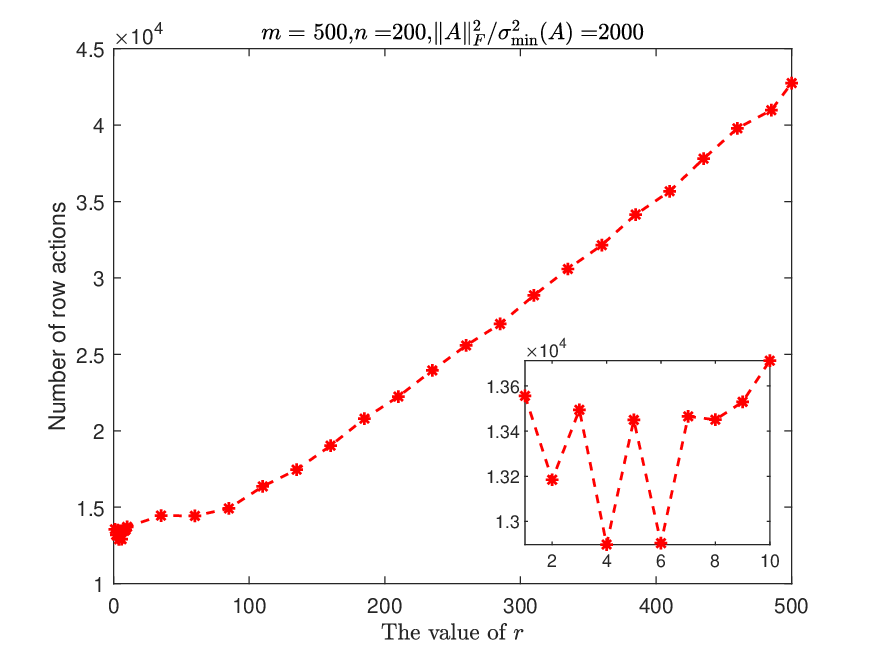}
	\end{tabular}
	\caption{Performance of mRrDR with different values of $r$.  $r=1$ refers to the mRK method. The parameters $\alpha=0.5, \beta=0.4$.  The title of each plot indicates the dimensions of the matrix $A$ and the value of $\|A\|^2_F/\sigma_{\min}^2(A)$. The number of row actions is employed to illustrate the evolutions for the different settings of mRrDR. }
	\label{figue6261}
\end{figure}

%\subsection{Comparison to the cyclic DR method}

\subsection{Comparison to the cyclic DR method}

In this subsection, we compare the mRrDR method to the cyclic DR method \eqref{cyc-DR} to demonstrate the effectiveness of randomization.
During the test, mRrDR is implemented with $r=2$ and $(\alpha,\beta)=(0.5,0)$ or $(\alpha,\beta)=(0.5,0.4)$. For the  cyclic DR method \eqref{cyc-DR}, we set $r=2$ and $\alpha=0.5$.
%Figure \ref{figue7} shows the numerical results on randomly generated data, real-world data, and the AC problem.
From Figure \ref{figue7}, we can observe the significant improvement in efficiency that randomization and momentum bring.

\begin{figure}[hptb]
	\centering
	\begin{tabular}{cc}
		\includegraphics[width=0.31\linewidth]{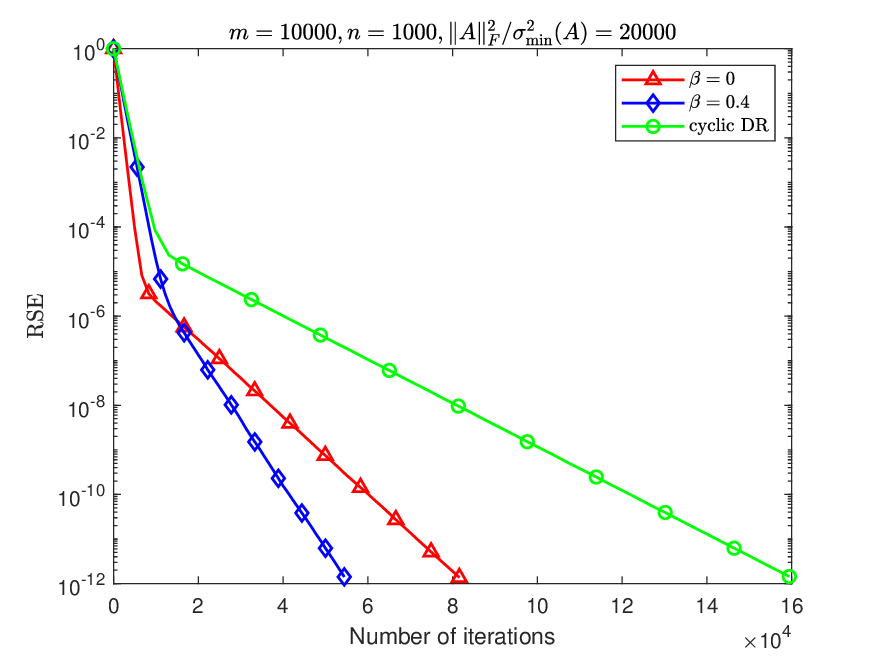}
		\includegraphics[width=0.31\linewidth]{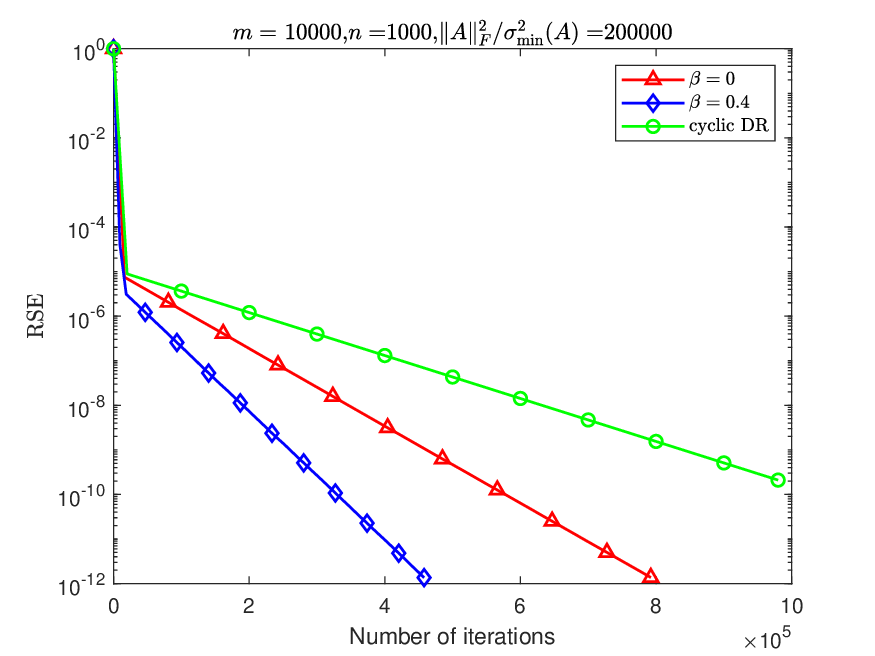}
		\includegraphics[width=0.31\linewidth]{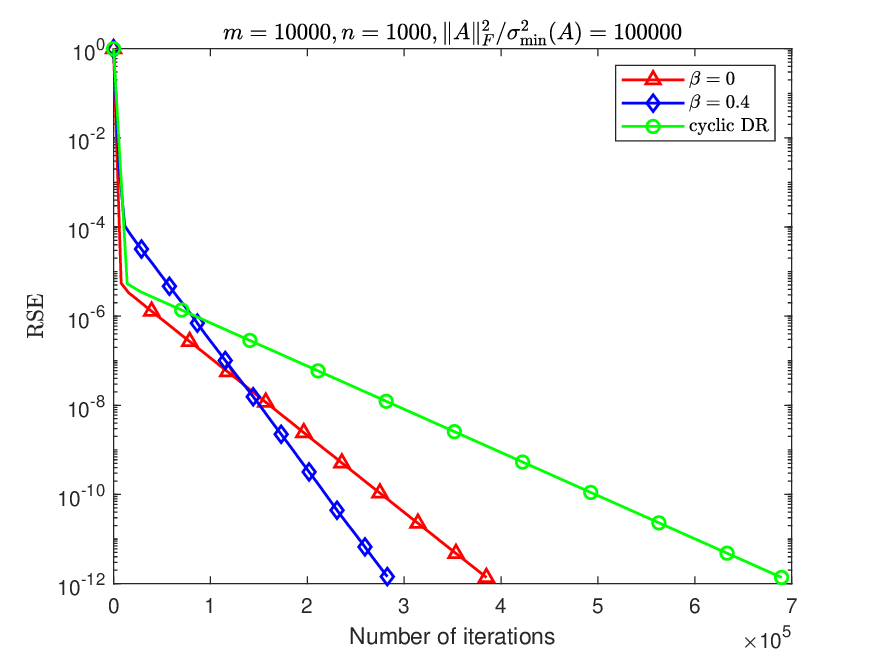}\\
		\includegraphics[width=0.31\linewidth]{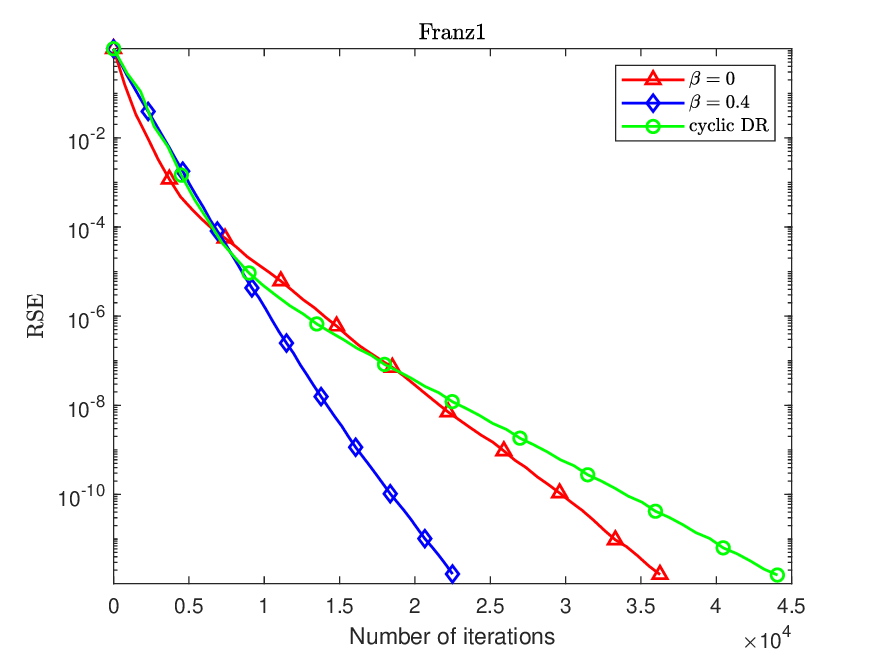}
		\includegraphics[width=0.31\linewidth]{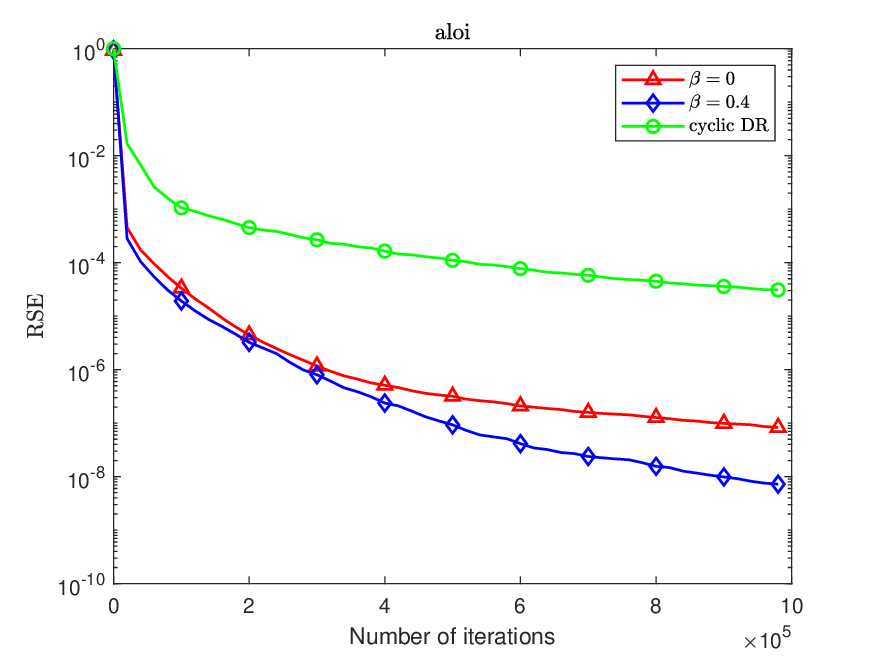}
		\includegraphics[width=0.31\linewidth]{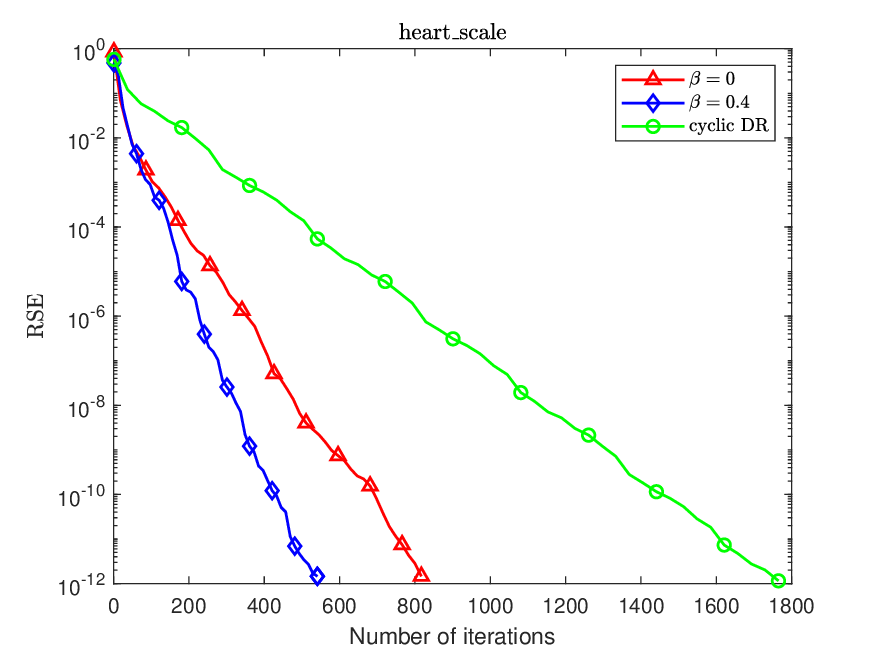}\\
		\includegraphics[width=0.31\linewidth]{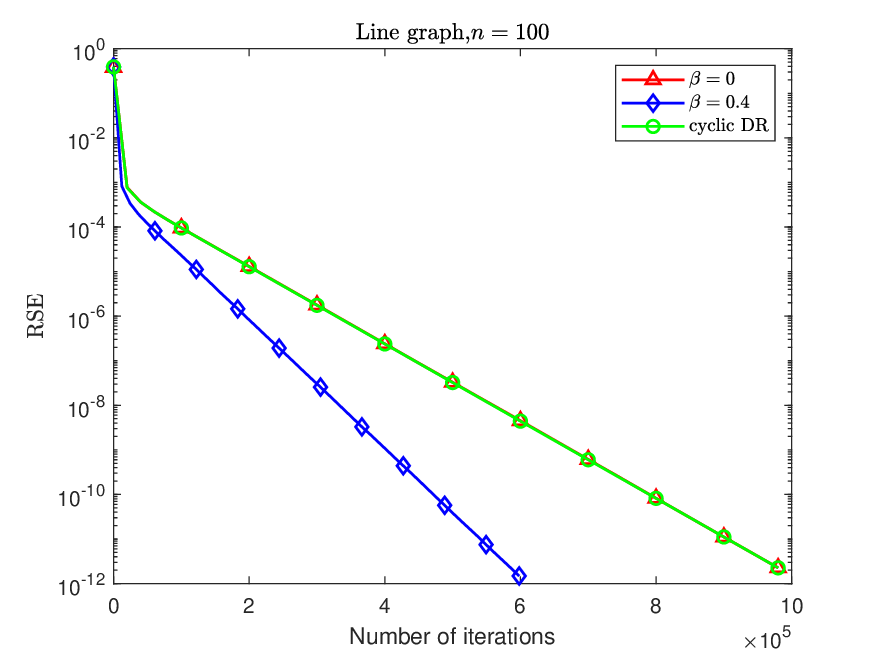}
		\includegraphics[width=0.31\linewidth]{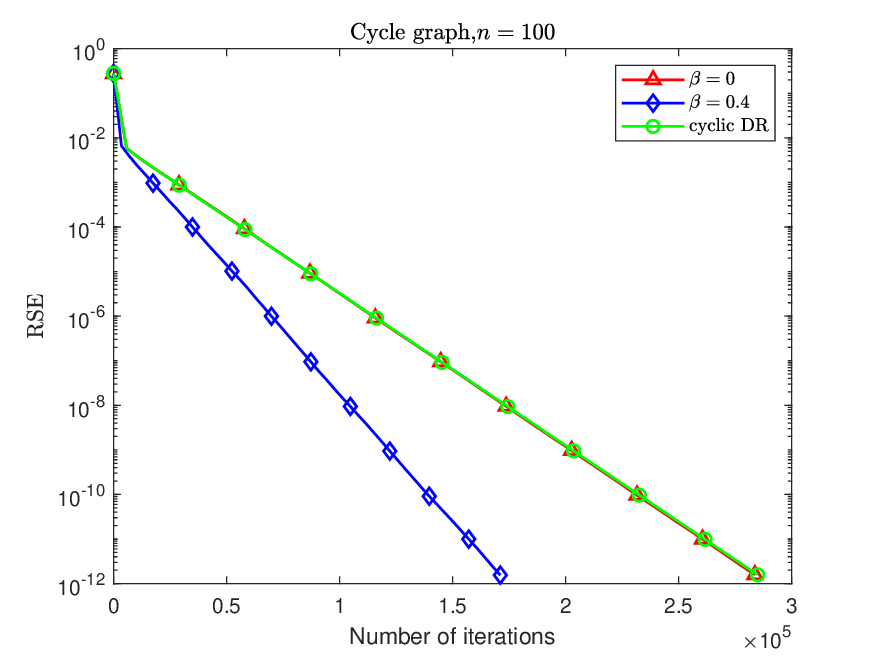}
		\includegraphics[width=0.31\linewidth]{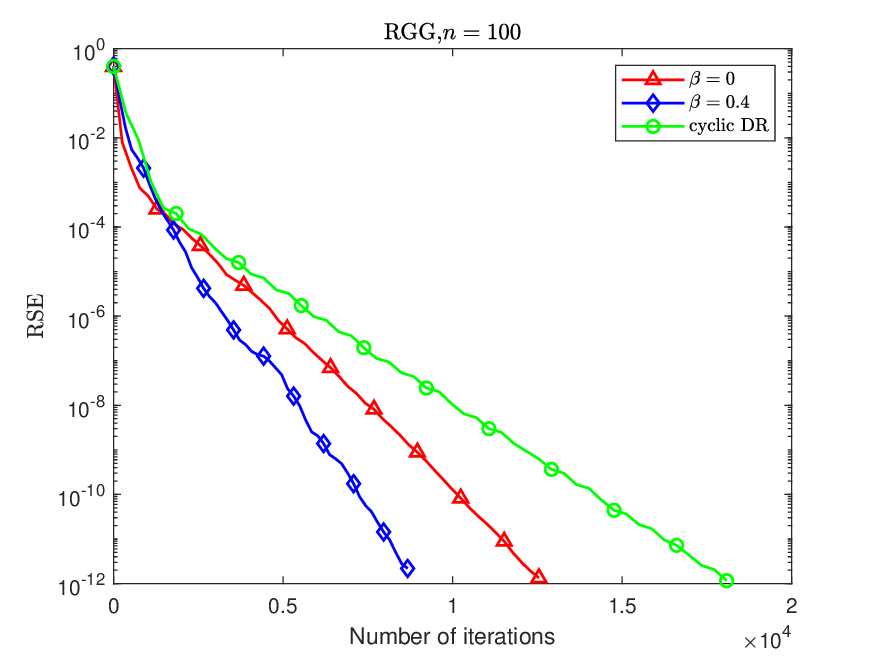}
	\end{tabular}
	\caption{Comparison of mRrDR $(r=2,\alpha=0.5)$ and the cyclic DR $(\alpha=0.5)$, where {\tt Franz1} from SuiteSparse Matrix Collection \cite{Kol19}, {\tt aloi} and {\tt heart-scale} from LIBSVM \cite{chang2011libsvm}. We stop the algorithms if $\operatorname{RSE}<10^{-12}$ or if the number of iteration exceeds a certain limit.}
	\label{figue7}
\end{figure}

\subsection{Comparison to the other methods}

We now compare mRrDR to other related methods for solving linear systems, including RK \cite{Str09},
randomized Gauss-Seidel (RGS) method \cite{Lev10,Gri12}, and randomly permuted alternating direction method of multipliers (RP-ADMM) \cite{Sun20}.
The RGS, also known as the randomized coordinate descent (RCD) method, updates with the following iterative strategy:
$$	x^{k+1}:=x^k-\frac{A_{j_{k}}^\top(Ax^{k}-b)}{\|A_{j_k}\|^2_2}e_{j_k},$$
where $j_k\in\{1,2,\ldots,n\}$ is selected with probability
$
\mbox{Pr}(j_k=j)=\frac{\|A_{j}\|_2^2}{\|A\|_F^2}
$, and $A_j, j=1,\ldots,n$ represent the columns of $A$ and $e_j$ is a column vector with the $j$-th entry being one and all other entries being zero.

The RP-ADMM tackles the following constrained problem
$$
\min\limits_{x\in\mathbb{R}^n} \ 0 \ \ \mbox{subject to} \ Ax=b,
$$
whose augmented Lagrangian function is defined by
$$
\mathcal{L}(x;\mu):=-\mu^\top(Ax-b)+\frac{\rho}{2}\|Ax-b\|^2_2,
$$
where $\rho>0$ is a given penalty parameter and $\mu$ denotes the Lagrangian multiplier.
Then the RP-ADMM method proceeds as follows: Given an approximation $(x^k,\mu^{k})$, it picks a permutation $\sigma$ of $\{1,\ldots,n\}$ uniformly at random, and constructs $(x^{k+1},\mu^{k+1})$ from $(x^k,\mu^{k})$ via
$$(x^{k+1})_{\sigma(i)}=\arg\min\limits_{x_{\sigma(i)}}\mathcal{L}((x^{k+1})_{\sigma(1)},\ldots,(x^{k+1})_{\sigma(i-1)},
x_{\sigma(i)},(x^{k})_{\sigma(i+1)},\ldots,(x^{k})_{\sigma(n)};\mu^{k})
$$
and
$
\mu^{k+1}=\mu^k-(Ax^{k+1}-b).
$
We note that the alternating direction method of multipliers (ADMM) is equivalent to the DR method \cite{Eck92,han2022survey} in the sense that the sequences generated by both algorithms coincide with a careful choice of starting point \cite[Remark 3.14]{Bau15Pro}.

Since there are seldom differences in the numerical performance for $ r\in[1,10]$, here we only show the cases where $r=2$.
For the RP-ADMM method, during our test, we set $\rho=1$ and the initial vector $\mu^0=0$.
Figures \ref{figue10} and \ref{figue11} summarize the results of the experiment.
It can be seen that  mRrDR  is more efficient than the other considered methods.

\begin{figure}[hptb]
	\centering
	\begin{tabular}{cc}
		\includegraphics[width=0.31\linewidth]{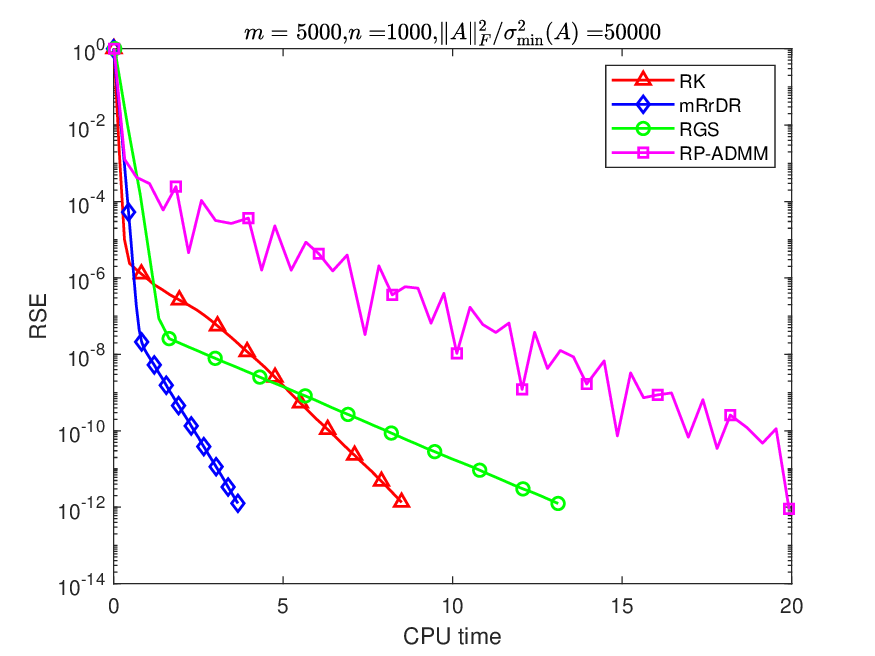}
		\includegraphics[width=0.31\linewidth]{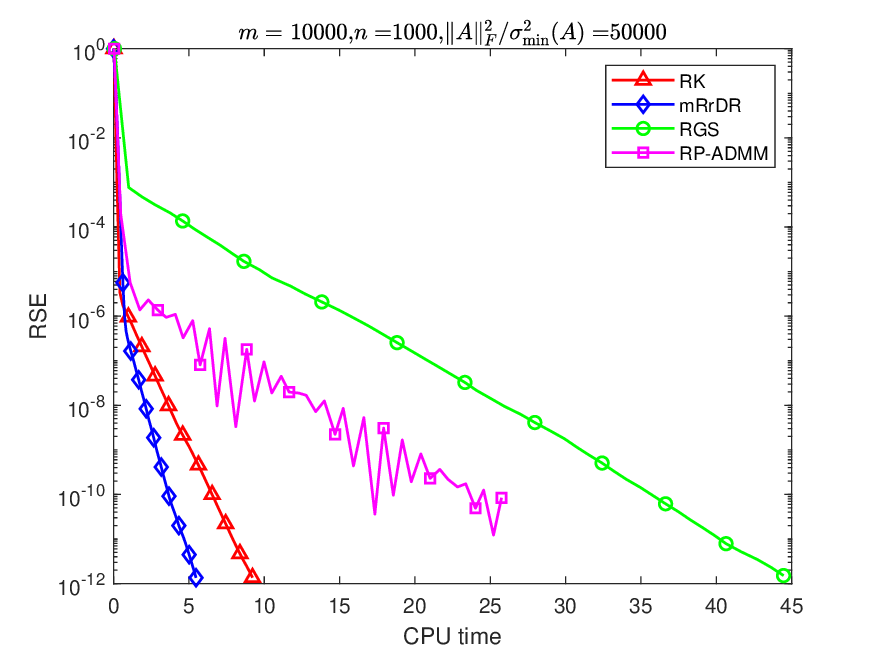}
		\includegraphics[width=0.31\linewidth]{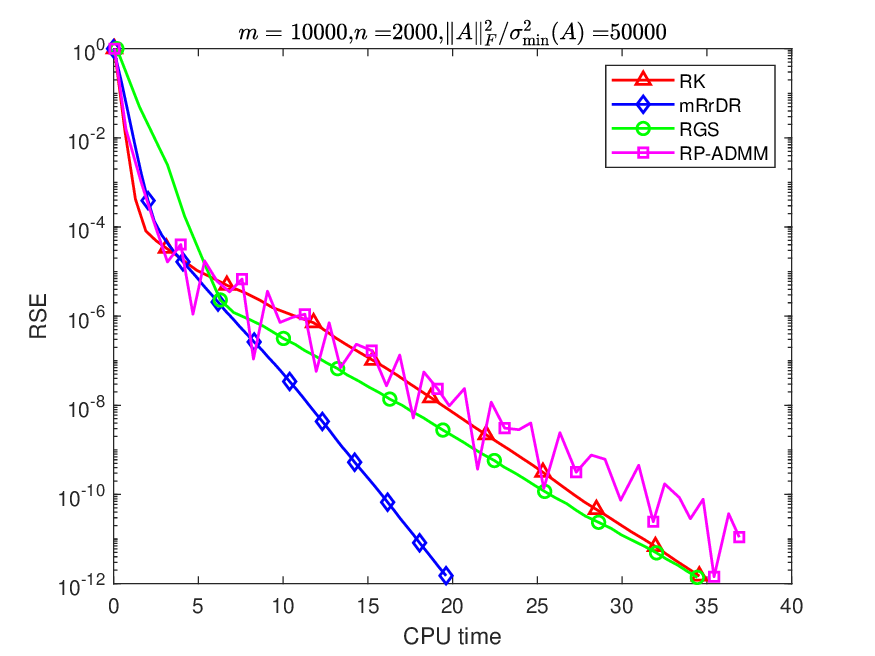}
	\end{tabular}
	\caption{Performance of RK, mRrDR, RGS, and RP-ADMM for synthetic data. We stop the algorithms if $\operatorname{RSE}<10^{-12}$ or if the number of iterations exceeds a certain limit.}
	\label{figue10}
\end{figure}

\begin{figure}[hptb]
	\centering
	\begin{tabular}{cc}
		\includegraphics[width=0.31\linewidth]{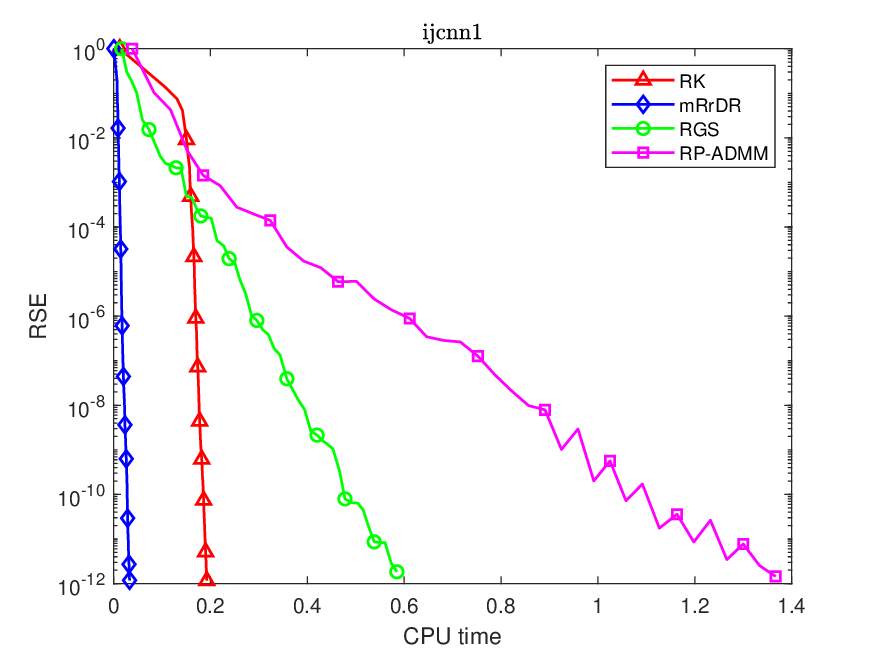}
		\includegraphics[width=0.31\linewidth]{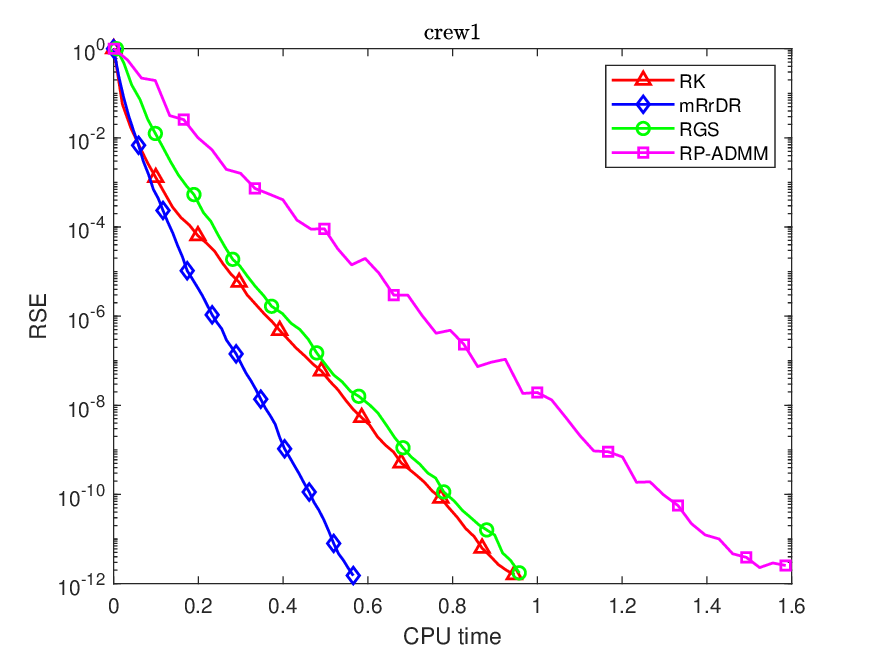}
		\includegraphics[width=0.31\linewidth]{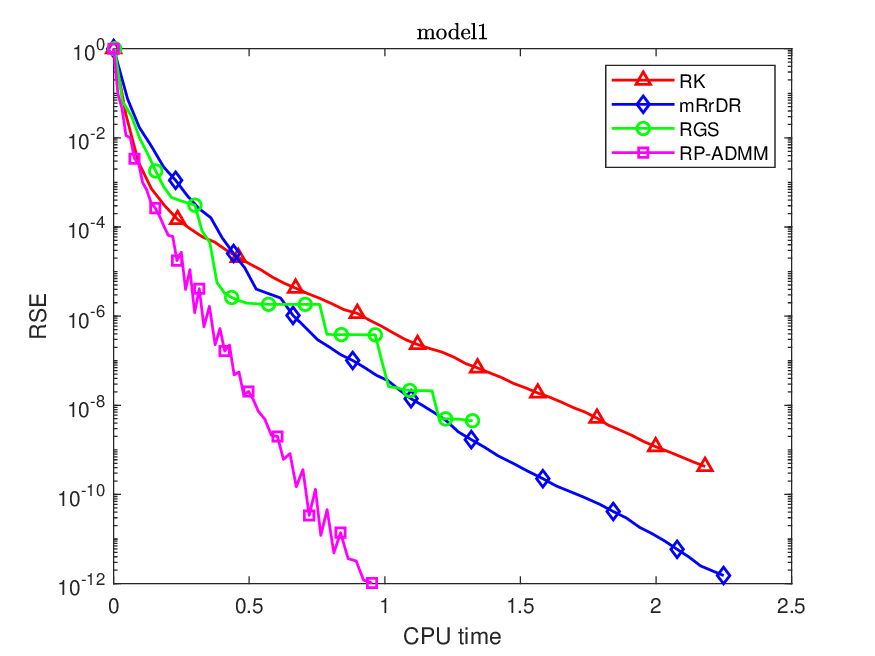}\\
		\includegraphics[width=0.31\linewidth]{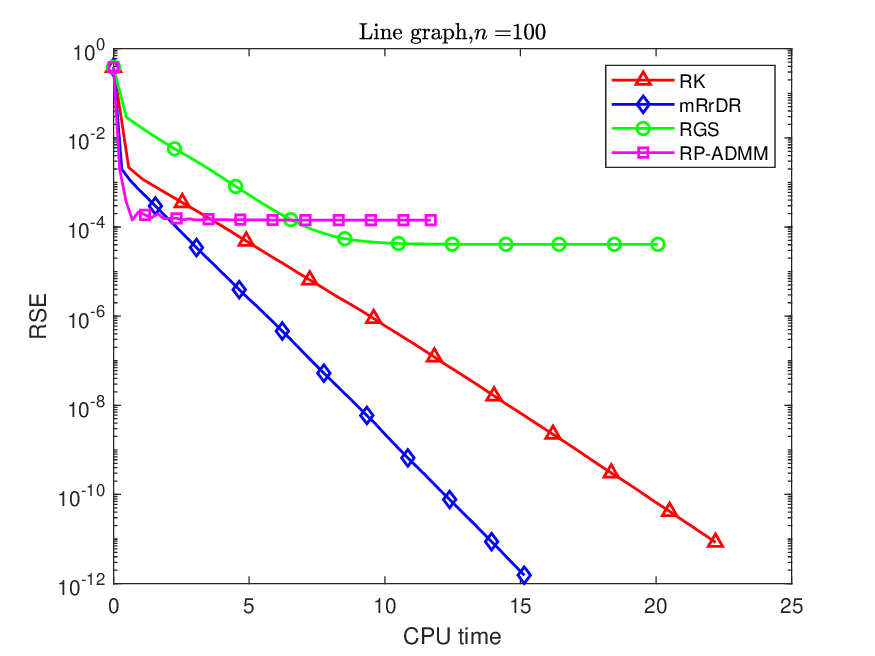}
		\includegraphics[width=0.31\linewidth]{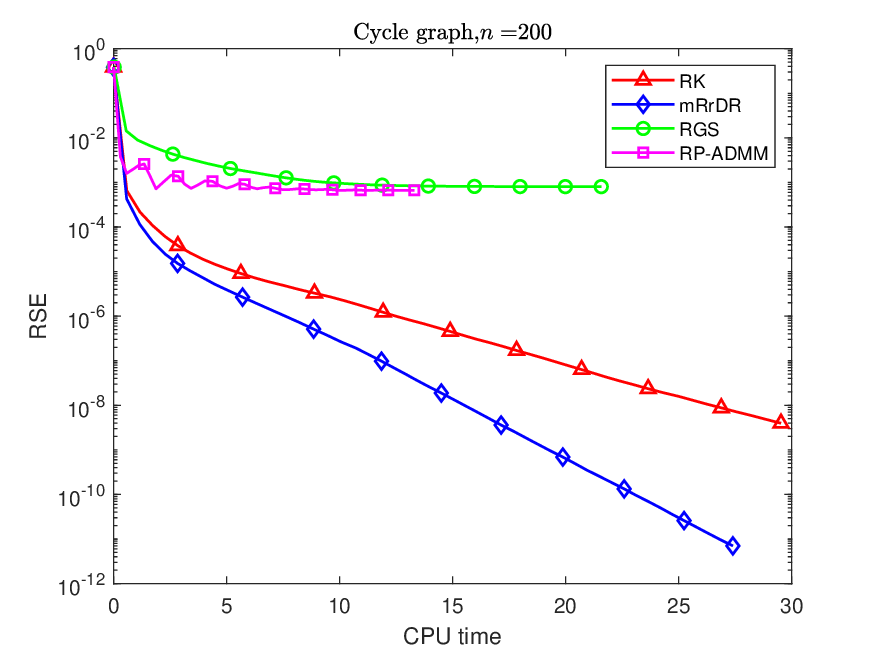}
		\includegraphics[width=0.31\linewidth]{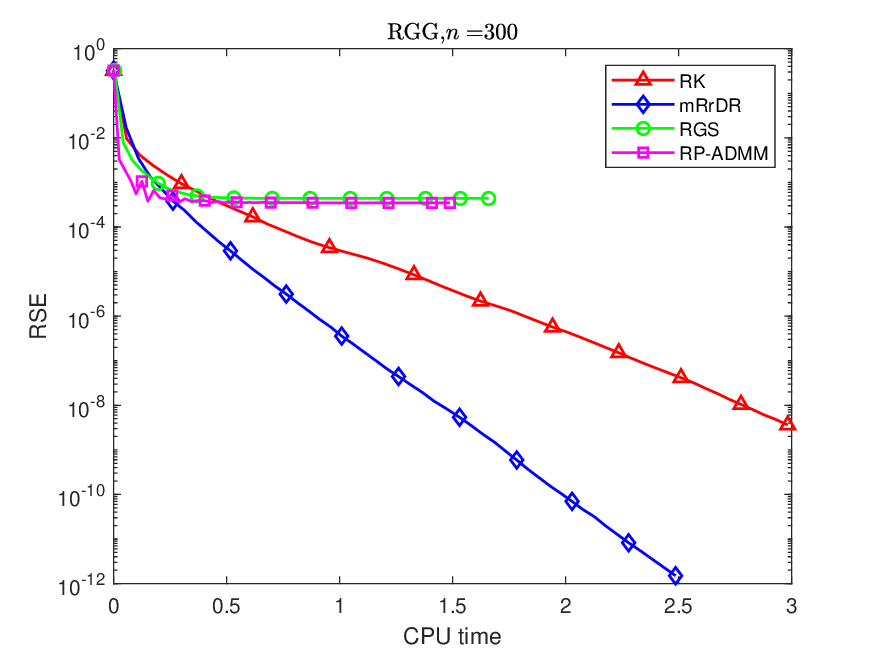}
	\end{tabular}
	\caption{Performance of RK, mRrDR, RGS, and RP-ADMM for real world data and the AC problem, where {\tt crew1} and {\tt model1} from SuiteSparse Matrix Collection \cite{Kol19},  {\tt ijcnn1} from LIBSVM \cite{chang2011libsvm}. We stop the algorithms if $\operatorname{RSE}<10^{-12}$ or if the number of iterations exceeds a certain limit.  }
	\label{figue11}
\end{figure}

\subsection{Comparison to {\tt pinv} and {\tt lsqminnorm}}
In this subsection,  we compare the performance of  mRrDR with  {\sc Matlab} functions  {\tt pinv} and {\tt lsqminnorm}.
To easily obtain the least-norm solution, we first generate full column rank coefficient matrices as follows.
For given $m\geq n$, and $\kappa>1$, we set $A=U D V^\top$, where $U \in \mathbb{R}^{m \times n}, D \in \mathbb{R}^{n\times n}$, and $V \in \mathbb{R}^{n \times n}$. Using {\sc Matlab}  notation, these matrices are generated by {\tt [U,$\sim$]=qr(randn(m,n),0)}, {\tt [V,$\sim$]=} {\tt qr(randn(n,n),0)}, and {\tt D=diag(1+($\kappa$-1).*rand(n,1))}. Note that the condition number of $A$ is now upper bounded by $\kappa$. Next, we generate the solution vector $x^*$ by setting $x^*={\tt randn(n,1)}$, and then we calculate $b=Ax^*$ to obtain the right-hand side vector of the linear system.  It can be observed that $x^*$ is the desired unique solution of the constructed linear system. %$\|\tilde{x}-A^{\dagger}b\|_2=O\left(2\times 10^{-14}\right)$

Figure \ref{figue12} illustrates our experimental results with fixed $n$. The mRrDR method is implemented with $r=1$ and $(\alpha,\beta)=(0.5,0.4)$, or $r=2$ and $(\alpha,\beta)=(0.5,0.4)$. We terminate the mRrDR method if the accuracy of its approximate solution is comparable to that of the approximate solution obtained using {\tt pinv} and {\tt lsqminnorm}.
In Figure \ref{figue12}, we plot the computing time  against the increasing number of rows. It can be observed that when the number of rows exceeds certain thresholds, mRrDR  outperforms {\tt pinv} and {\tt lsqminnorm}.
We can also find that the performance of the mRrDR method is more sensitive to the increase of the condition number $\kappa$, as the convergence bound implies.

\begin{figure}[hptb]
	\centering
	\begin{tabular}{cc}
		\includegraphics[width=0.31\linewidth]{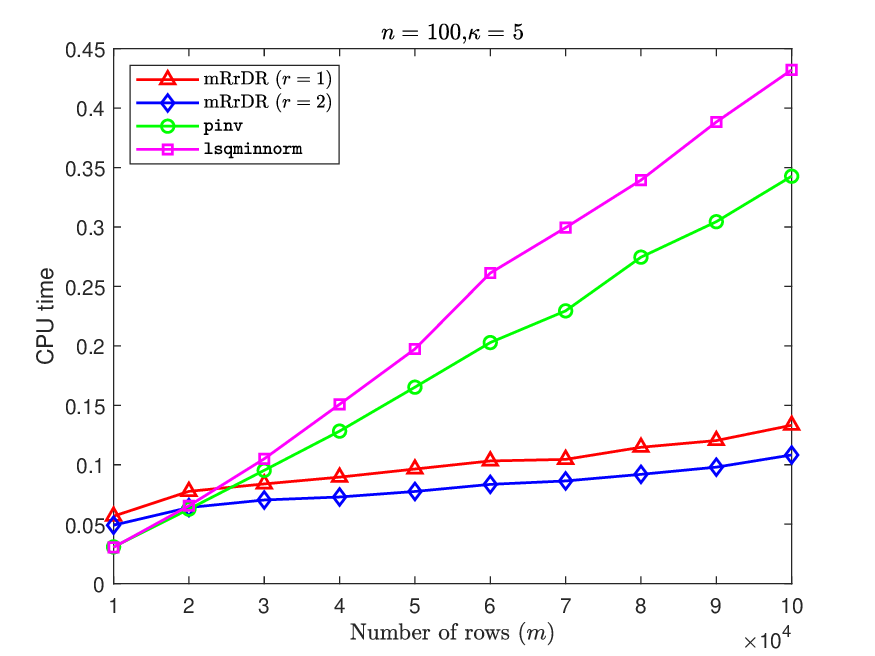}
		\includegraphics[width=0.31\linewidth]{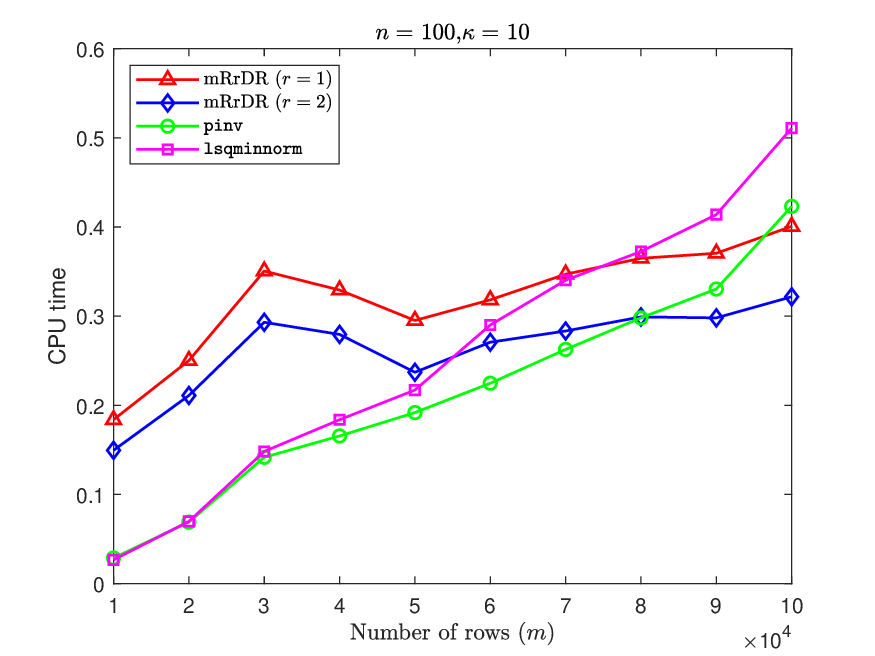}
		\includegraphics[width=0.31\linewidth]{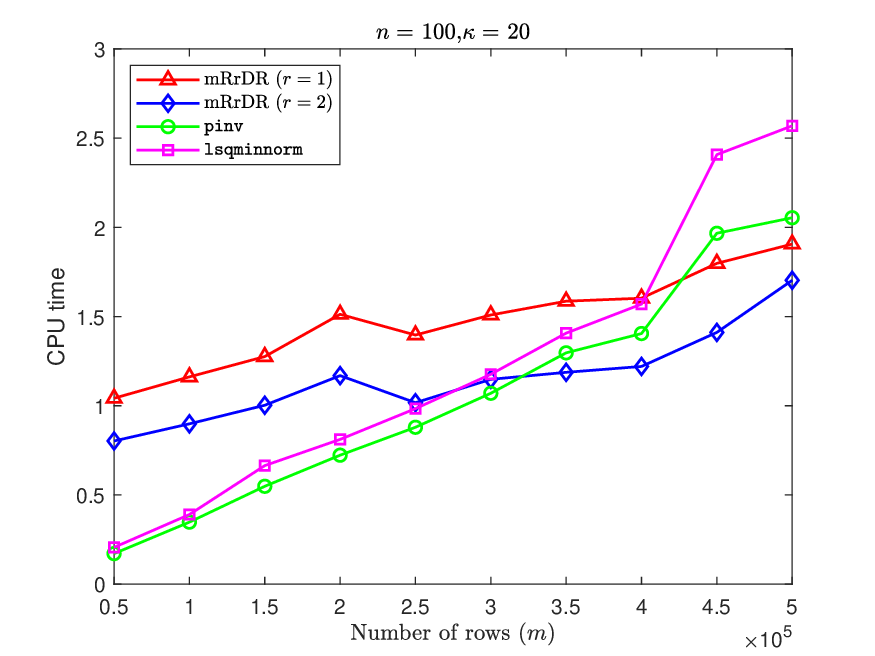}
	\end{tabular}
	\caption{Figures depict the CPU time (in seconds) vs increasing number of rows.  The title of each plot indicates the values of $n$ and $\kappa$. }
	\label{figue12}
\end{figure}

\section{Concluding remarks}\label{sec:conc}

In this work, we studied the $r$-sets-Douglas-Rachford method enriched with randomization and heavy ball momentum for solving linear systems.  We proved global linear convergence rates of the method as well as an accelerated linear rate in terms of the norm of expected error. Our convergence analysis showed the effectiveness of  randomization in  simplifying the analysis  of the DR method and making the divergent $r$-sets-DR method converge linearly.
We corroborated our theoretical results with extensive experimental testing and confirmed the better performance of the mRrDR method. %In addition, we arrived at the conclusion that $ (\alpha,\beta) = (0.5,0.4) $ is a relatively good choice  for  a satisfactory convergence of mRrDR.

There are still many possible future venues of research.
A bunch of advanced schemes for the selection of sets to project have been investigated in the literature of the Kaczamarz method, such as the greedy selection rule \cite{Bai18Gre}, its weighted variant in \cite{Ste20Wei} and the approach with sampling in \cite{De17}. These criteria are convenient to be adopted to the DR context for further improvement in efficiency. Moreover, the linear systems arising in practical problems are very likely to be inconsistent due to noise, which contradicts the basic assumption in this paper. The extended randomized Kaczmarz \cite{Zou12,Du20Ran} was proposed for such cases. It should also be a valuable topic to explore the extensions of DR methods for inconsistent linear systems.

\bibliographystyle{plain}
\bibliography{references}

\section{Appendix. Proof of the main results}
	\label{sec:appd}

For any $i\in\{1,\ldots,m\}$, we set
$$
T_{C_i}:=I-2\frac{a_{i}a_i^{\top}}{\|a_i\|^2_2}.
$$
It is easy to verify that $T_{C_i}T_{C_i}=I$, and thus for any $y\in \mathbb{R}^n$, $\|T_{C_i}y\|_2=\|y\|_2$.

\subsection{Proof of Theorems \ref{main-ThmrRDR} and \ref{THMfmm}}
We first prove some crucial facts about the algorithms.

\begin{lemma}\label{lemma-61}
	Let $\{z^k_r\}_{k=0}^{\infty}$ be the iteration sequence generated by Algorithm \ref{r-RDRK} or Algorithm \ref{r-mRDRK}.
	Then
	$$
	\|z^k_r-x^*\|_2=\|x^k-x^*\|_2.
	$$
\end{lemma}
\begin{proof}
	Note that
	\begin{equation}\label{prf-lemma1-1}
		\begin{aligned}
			z^k_r-x^*=&z^k_{r-1}-x^*-2\frac{\langle a_{j_{k_{r}}},z_{r-1}^k\rangle-b_{j_{k_{r}}}}{\|a_{j_{k_r}}\|^2_2}a_{j_{k_{r}}}
			=z^k_{r-1}-x^*-2\frac{\langle a_{j_{k_{r}}},z_{r-1}^k-x^*\rangle}{\|a_{j_{k_r}}\|^2_2}a_{j_{k_{r}}}
			\\
			=&\left(I-2\frac{a_{j_{k_r}}a_{j_{k_r}}^{\top}}{\|a_{j_{k_r}}\|^2_2}\right)(z^k_{r-1}-x^*)
			=T_{C_{j_{k_r}}}(z^k_{r-1}-x^*)
			\\
			=&T_{C_{j_{k_r}}}T_{C_{j_{k_{r-1}}}}\cdots T_{C_{j_{k_1}}}(z^k_{0}-x^*)
			= T_{C_{j_{k_r}}}T_{C_{j_{k_{r-1}}}}\cdots T_{C_{j_{k_1}}}(x^k-x^*).
		\end{aligned}
	\end{equation}
	Note that for any $y\in\mathbb{R}^n$, it holds that $\|y\|_2=\|T_{C_{j_{k_r}}}T_{C_{j_{k_{r-1}}}}$ $\cdots T_{C_{j_{k_1}}}y\|_2$.
	Hence, we have
	$$
	\|z^k_r-x^*\|_2=\|T_{C_{j_{k_r}}}T_{C_{j_{k_{r-1}}}}\cdots T_{C_{j_{k_1}}}(x^k-x^*)\|_2=\|x^k-x^*\|_2
	$$
	as desired.
\end{proof}

\begin{lemma}
	\label{lemma-main-rRDR}
	Let $\{x^k\}_{k=0}^{\infty}$ and $\{z^k_r\}_{k=0}^{\infty}$ be the sequences generated by Algorithm \ref{r-RDRK} or Algorithm \ref{r-mRDRK}. Then
	$$
	\|(1-\alpha)x^{k}+\alpha z_r^k-x^*\|^2_2=\big(\alpha^2+(1-\alpha)^2\big)\|x^k-x^*\|^2_2
	+2\alpha(1-\alpha)\langle z^k_r-z^*,x^k-x^*\rangle.
	$$
\end{lemma}
\begin{proof}
	We have
	$$		\begin{aligned}
		\|(1-\alpha)x^{k}+\alpha z_r^k-x^*\|_2^2=&\|(1-\alpha)(x^{k}-x^*)+\alpha (z_r^k-x^*)\|_2^2
		\\
		=&\|(1-\alpha)(x^{k}-x^*)\|^2_2+2\alpha(1-\alpha)\langle x^{k}-x^*,z_r^k-x^*\rangle
		+\|\alpha(z^k_r-x^*)\|^2_2
		\\
		=&\big(\alpha^2+(1-\alpha)^2\big)\|x^k-x^*\|^2_2+2\alpha(1-\alpha)\langle z_r^k-x^*,x^k-x^*\rangle,
	\end{aligned}
	$$
	where the last equality follows from Lemma \ref{lemma-61}.
\end{proof}

\begin{lemma}
	\label{lemma-inrowspace}
	Let $\{x^k\}_{k=0}^{\infty}$ be the sequences generated by Algorithm \ref{r-RDRK} or Algorithm \ref{r-mRDRK} $(x^1=x^0)$ and $x_0^*=A^{\dagger}b+(I-A^\dagger A)x^0$. Then $ x^k - x^*_0 \in \operatorname{Row}(A) $.
\end{lemma}
\begin{proof}
	Since Algorithm \ref{r-RDRK} is a special case of Algorithm \ref{r-mRDRK} with $ \beta = 0 $, we only concentrate on the proof for Algorithm \ref{r-mRDRK}.
	We first prove that $ x^k \in x^0 + \operatorname{Row}(A) $ for any $ k\geq0 $ by induction.
	Initially, $ x^1 = x^0 \in x^0 + \operatorname{Row}(A) $. If $ x^k \in x^0 + \operatorname{Row}(A) $ holds for all $ k \leq t(t \geq 1) $, then
	$$
	z^t_r  = x^t - 2\frac{\langle a_{j_{t_{1}}},x^t-x^*\rangle}{\|a_{j_{t_1}}\|^2_2}a_{j_{t_{1}}} - \cdots -  2\frac{\langle a_{j_{t_{r}}},z_{r-1}^t-x^*\rangle}{\|a_{j_{t_r}}\|^2_2}a_{j_{t_{r}}} \in x^0 + \operatorname{Row}(A),
	$$
	and $
	x^t - x^{t-1} \in \operatorname{Row}(A).
	$
	Therefore,
	$$\begin{aligned}
		x^{t+1} & = (1-\alpha) x^t+\alpha z_{r}^t+\beta(x^t-x^{t-1}) \\
		& \in (1-\alpha)(x^0 + \operatorname{Row}(A)) + \alpha(x^0 + \operatorname{Row}(A)) + \operatorname{Row}(A) \\
		& = x^0 + \operatorname{Row}(A).
	\end{aligned}
	$$
	Thus $ x^k \in x^0 + \operatorname{Row}(A) $ holds for all $k\geq 0$. Note that $ x_0^* =A^{\dagger}b+(I-A^\dagger A)x^0 = A^{\dagger}(b-Ax^0) + x^0 \in x^0 + \operatorname{Row}(A)$, we arrive at the conclusion $ x^k - x^*_0 \in \operatorname{Row}(A) $.
\end{proof}

Now we are ready to prove the main results.  In fact, Theorem \ref{main-ThmrRDR} can be directly derived from Theorem \ref{THMfmm} by letting $\beta=0$. %We include a simple proof of Theorem \ref{main-ThmrRDR} here for ease of reading.
We include an individual proof of Theorem \ref{main-ThmrRDR} for readability, meanwhile showing the tightness of the convergence rate in a clear way.
%We include an individual proof of Theorem \ref{main-ThmrRDR} to show the tightness of the convergence rate.

\begin{proof}[Proof of Theorem \ref{main-ThmrRDR}]
	From Lemma \ref{lemma-main-rRDR} and taking the conditional expectation under the probability $\mbox{Pr}(\mbox{row}=j_{k_{\ell}})=\frac{\|a_{j_{k_{\ell}}}\|^2_2}{\|A\|_{F}^2}$, we get
	\begin{equation}\label{xk1-x}
		\begin{aligned}
			&\mathop{\mathbb{E}}\limits_{j_{k_r},\ldots,j_{k_1}}\big[\|x^{k+1}-x_{0}^*\|^2_2|x^k\big]=
			\mathop{\mathbb{E}}\limits_{j_{k_r},\ldots,j_{k_1}}\big[\|(1-\alpha)x^{k}+\alpha z_r^k-x_{0}^*\|^2_2|x^k\big]
			\\
			=&\left(\alpha^2+(1-\alpha)^2\right)\|x^k-x_{0}^*\|^2_2
			+2\alpha(1-\alpha)
			\mathop{\mathbb{E}}\limits_{j_{k_r},\ldots,j_{k_1}}\left[\left\langle T_{C_{j_{k_r}}}T_{C_{j_{k_{r-1}}}}\ldots T_{C_{j_{k_1}}}(x^k-x_{0}^*),x^k-x_{0}^*\right\rangle\right]\\
			=&\left(\alpha^2+(1-\alpha)^2\right)\|x^k-x_{0}^*\|^2_2 +2\alpha(1-\alpha)\left\langle \mathop{\mathbb{E}}\limits_{j_{k_r}}[T_{C_{j_{k_r}}}] \ldots \mathop{\mathbb{E}}\limits_{j_{k_1}}[T_{C_{j_{k_1}}}](x^k-x_{0}^*),x^k-x_{0}^* \right\rangle\\
			=&\left(\alpha^2+(1-\alpha)^2\right)\|x^k-x_{0}^*\|^2_2  +2\alpha(1-\alpha)\left\langle \left(I-2\frac{A^\top A}{\|A\|^2_F}\right)^r(x^k-x_{0}^*),x^k-x_{0}^* \right\rangle.
		\end{aligned}
	\end{equation}
	The first equality follows from Step $6$ in Algorithm \ref{r-RDRK}; the second equality follows from \eqref{prf-lemma1-1}; the third equality follows from the  linearity of the expectation and the independence of $j_{k_1},\ldots,j_{k_r}$; the last equality follows from the fact that for any $\ell\in\{1,\ldots,r\}$, %we have
	\begin{equation}\label{prf-equ613-1}
		\mathop{\mathbb{E}}\limits_{j_{k_\ell}}\left[T_{C_{j_{k_\ell}}}\right]=\sum\limits_{i=1}^m\frac{\|a_i\|^2_2}{\|A\|^2_F}
		\left(I-2\frac{a_ia_i^\top}{\|a_i\|^2_2}\right)
		=I-2\frac{A^\top A}{\|A\|^2_F}.
	\end{equation}
	By Lemma \ref{lemma-inrowspace}, we know that $x^k-x_{0}^*\in\mbox{Row}(A)=\mbox{Range}(A^\top)$.
	Then we have
	\begin{equation}\label{prf-equ1}
		(x^k-x_{0}^* )^\top \left(I-2\frac{A^\top A}{\|A\|^2_F}\right)^r(x^k-x_{0}^*)\leq \left(1-2\frac{\sigma_{\min}^2(A)}{\|A\|^2_F}\right)^r\|x^k-x_{0}^*\|^2_2,
	\end{equation}
	which implies
	$$
	\mathop{\mathbb{E}}\limits_{j_{k_r},\ldots,j_{k_1}}\big[\|x^{k+1}-x_{0}^*\|^2_2|x^k\big]\leq \left(\alpha^2+(1-\alpha)^2+2\alpha(1-\alpha)\left(1-2\frac{\sigma_{\min}^2(A)}{\|A\|^2_F}\right)^r\right)\|x^k-x_{0}^*\|^2_2.
	$$
	Taking the expectation over the entire history we have
	$$
	\mathop{\mathbb{E}} [ \|x^{k+1}-x_{0}^*\|^2_2 ] \leq \left(\alpha^2+(1-\alpha)^2+2\alpha(1-\alpha)\left(1-2\frac{\sigma_{\min}^2(A)}{\|A\|^2_F}\right)^r\right)\mathop{\mathbb{E}} \left[ \|x^k-x_{0}^*\|^2_2\right].
	$$
	By induction on the iteration index $k$, we can obtain the desired result.
\end{proof}

\begin{remark}\label{xie-tight}
	If $\sigma_1(A) = \sigma_{\min}(A)$, that is, all nonzero singular values of $A$ are equal, then the inequality in \eqref{prf-equ1} becomes equality. As a result, the upper bound in Theorem \ref{main-ThmrRDR} is also equality, indicating that the upper bound in Theorem \ref{main-ThmrRDR} is tight.
\end{remark}

To prove Theorem \ref{THMfmm}, the following result is required.
\begin{lemma}[\cite{han2022pseudoinverse}, Lemma $8.1$]\label{lemma-key}
	Fix $F^1=F^0\geq 0$ and let $\{F^k\}_{k\geq 0}$ be a sequence of nonnegative real numbers satisfying the relation
	$$
	F^{k+1}\leq \gamma_1 F^k+\gamma_2F^{k-1},\ \ \forall \ k\geq 1,
	$$
	where $\gamma_2\geq0,\gamma_1+\gamma_2<1$ and at least one of the coefficients $\gamma_1,\gamma_2$ is positive. Then the sequence satisfies the relation
	$$F^{k+1}\leq q^k(1+\tau)F^0,\ \ \forall \ k\geq 0,$$
	where $q=\frac{\gamma_1+\sqrt{\gamma_1^2+4\gamma_2}}{2}$ and $\tau=q-\gamma_1\geq 0$. Moreover,
	$
	q\geq \gamma_1+\gamma_2,
	$
	with equality if and only if $\gamma_2=0$.
\end{lemma}

Now, we are going to prove Theorem \ref{THMfmm}.

\begin{proof}[Proof of Theorem \ref{THMfmm}]
	First, we have
	\begin{equation}\label{prf-THMm1}
		\begin{aligned}
			&\|x^{k+1}-x_{0}^*\|_2^2=\|(1-\alpha) x^{k}+\alpha z^{k}_r+\beta(x^k-x^{k-1})-x_{0}^*\|^2_2\\
			= &\underbrace{\|(1-\alpha) x^{k}+\alpha z^k_r-x_{0}^*\|^2_2}_{\textcircled{a}}+\underbrace{2\beta\langle(1-\alpha) x^{k}+\alpha z^k_r-x_{0}^*,x^k-x^{k-1} \rangle}_{\textcircled{b}}
			+\underbrace{\beta^2\|x^k-x^{k-1}\|^2_2}_{\textcircled{c}}.
		\end{aligned}
	\end{equation}
	We now analyze the three expressions $\textcircled{a},\textcircled{b},\textcircled{c}$ separately. From Lemma \ref{lemma-main-rRDR}, we have
	$$\textcircled{a}=\big(\alpha^2+(1-\alpha)^2\big)\|x^k-x_{0}^*\|^2_2+2\alpha(1-\alpha)\langle z^r_k-x_{0}^*,x^k-x_{0}^*\rangle.
	$$
	We now bound the second expression. First, we have
	$$\begin{aligned}
		\textcircled{b}=&2(1-\alpha) \beta\langle x^{k}-x_{0}^*,x^k-x^{k-1} \rangle+2\alpha\beta\langle z^k_r-x_{0}^*,x^k-x^{k-1} \rangle
		\\
		=&2(1-\alpha)\beta \langle x^{k}-x_{0}^*,x^k-x_{0}^* \rangle+2(1-\alpha)\beta \langle x^{k}-x_{0}^*,x_{0}^*-x^{k-1} \rangle
		+2\alpha\beta\langle z^k_r-x_{0}^*,x^k-x^{k-1} \rangle\\
		=&2(1-\alpha) \beta\|x^{k}-x_{0}^*\|^2_2+2(1-\alpha)\beta \langle x^{k}-x_{0}^*,x_{0}^*-x^{k-1} \rangle
		+2\alpha\beta\langle z^k_r-x_{0}^*,x^k-x^{k-1} \rangle.
	\end{aligned}$$
	Noting that
	$ 2\langle x^{k}-x_{0}^*,x_{0}^*-x^{k-1} \rangle\leq \|x^k-x_{0}^*\|^2_2+\|x^{k-1}-x_{0}^*\|^2_2, $
	which implies
	$$\textcircled{b}\leq 3(1-\alpha)\beta \|x^{k}-x_{0}^*\|^2_2+(1-\alpha) \beta \|x^{k-1}-x_{0}^*\|^2_2+2\alpha\beta\langle z^k_r-x_{0}^*,x^k-x^{k-1} \rangle.$$
	The third expression can be bounded by
	$$\textcircled{c}\leq2\beta^2 \|x^{k}-x_{0}^*\|^2_2+2\beta^2\|x^{k-1}-x_{0}^*\|^2_2.$$
	By substituting all the bounds  into \eqref{prf-THMm1}, we obtain
	$$\begin{aligned}
		\|x^{k+1}-x_{0}^*\|_2^2 &
		\leq \big(\alpha^2+(1-\alpha)^2\big)\|x^k-x_{0}^*\|^2_2+2\alpha(1-\alpha)\langle z^k_r-x_{0}^*,x^k-x_{0}^*\rangle\\
		& \quad +3(1-\alpha) \beta \|x^{k}-x_{0}^*\|^2_2+(1-\alpha)\beta \|x^{k-1}-x_{0}^*\|^2_2+2\alpha\beta\langle z^k_r-x_{0}^*,x^k-x^{k-1} \rangle
		\\
		& \quad +2\beta^2 \|x^{k}-x_{0}^*\|^2_2+2\beta^2\|x^{k-1}-x_{0}^*\|^2_2
		\\
		&\leq  \big(\alpha^2+(1-\alpha)^2+3(1-\alpha)\beta+2\beta^2\big)\|x^k-x_{0}^*\|^2_2+\big((1-\alpha)\beta+2\beta^2\big)\|x^{k-1}-x_{0}^*\|^2_2
		\\
		& \quad +2\alpha(1-\alpha)\langle z^k_r-x_{0}^*,x^k-x_{0}^*\rangle+2\alpha\beta\langle z^k_r-x_{0}^*,x^k-x^{k-1} \rangle
	\end{aligned}$$
	Now taking the conditional expectation under the probability $\mbox{Pr}(\mbox{row}=j_{k_{\ell}})=\frac{\|a_{j_{k_{\ell}}}\|^2_2}{\|A\|_{F}^2}$, we get
	\begin{equation}\label{xie-e-0909}
		\begin{aligned}
			\mathbb{E}\big[\|x^{k+1}-x_{0}^*\|^2_2|x^k\big]
			\leq& \big(\alpha^2+(1-\alpha)^2+3(1-\alpha)\beta+2\beta^2\big)\|x^k-x_{0}^*\|^2_2\\
			&+\big((1-\alpha)\beta+2\beta^2\big)\|x^{k-1}-x_{0}^*\|^2_2
			\\
			&+\underbrace{2\alpha(1-\alpha)\left\langle \left(I-2\frac{A^\top A}{\|A\|^2_F}\right)^r(x^k-x_{0}^*),x^k-x_{0}^* \right\rangle}_{\textcircled{d}}  \\
			& \
			+\underbrace{2\alpha\beta\left\langle \left(I-2\frac{A^\top A}{\|A\|^2_F}\right)^r(x^k-x_{0}^*),x^{k}-x^{k-1} \right\rangle}_{\textcircled{e}}.
	\end{aligned}\end{equation}
	Similar to the argument in \eqref{prf-equ1}, we know that
	$$\textcircled{d}\leq 2\alpha(1-\alpha)\left(1-2\frac{\sigma_{\min}^2(A)}{\|A\|^2_F}\right)^r\|x^k-x_{0}^*\|^2_2.
	$$
	For expression $\textcircled{e}$,  we have
	$$
	\begin{aligned}
		\textcircled{e}&=2\alpha\beta\left\langle \left(I-2\frac{A^\top A}{\|A\|^2_F}\right)^r(x^k-x_{0}^*),x^{k}-x_{0}^* \right\rangle+2\alpha\beta\left\langle \left(I-2\frac{A^\top A}{\|A\|^2_F}\right)^r(x^k-x_{0}^*),x_{0}^*- x^{k-1}\right\rangle\\
		&\leq3\alpha\beta\left\|\left(I-2\frac{A^\top A}{\|A\|^2_F}\right)^{r/2}(x^k-x_{0}^*)\right\|^2_2
		+\alpha\beta\left\|\left(I-2\frac{A^\top A}{\|A\|^2_F}\right)^{r/2}(x_{0}^*-x^{k-1})\right\|^2_2
		\\
		&\leq 3\alpha\beta\left(1-2\frac{\sigma_{\min}^2(A)}{\|A\|^2_F}\right)^r\|x^k-x_{0}^*\|^2_2+
		\alpha\beta\left(1-2\frac{\sigma_{\min}^2(A)}{\|A\|^2_F}\right)^r\|x^{k-1}-x_{0}^*\|^2_2,
	\end{aligned}
	$$
	where the first inequality follows from $2\langle v,u\rangle\leq \|v\|^2_2+\|u\|^2_2$ and
	the second inequality follows from the fact that $x^k-x_{0}^*\in\mbox{Row}(A)$, $x^{k-1}-x_{0}^*\in\mbox{Row}(A)$, and \eqref{prf-equ1}.
	Hence
	$$\begin{aligned}
		&\mathbb{E}\big[\|x^{k+1}-x_{0}^*\|^2_2|x^k\big] \leq \underbrace{\left(2\beta^2+(1-\alpha)\beta+\alpha\beta\left(1-2\frac{\sigma_{\min}^2(A)}{\|A\|^2_F}\right)^r\right)}_{\gamma_2}
		\|x^{k-1}-x_{0}^*\|^2_2		\\
		& +  \underbrace{\left(\alpha^2+(1-\alpha)^2+\left(2\alpha(1-\alpha)+3\alpha\beta\right)\left(1-2\frac{\sigma_{\min}^2(A)}{\|A\|^2_F}\right)^r
			+2\beta^2+3(1-\alpha)\beta
			\right)}_{\gamma_1}\|x^k-x_{0}^*\|^2_2.
	\end{aligned}
	$$
	By taking expectation over the entire history, and letting $F^k:=\mathbb{E}[\|x^k-x_{0}^*\|^2_2]$, we get the relation
	\begin{equation}\label{prf-THMm10}
		F^{k+1}\leq \gamma_1F^k+\gamma_2F^{k-1}.
	\end{equation}
	Noting that the requirements for the coefficients $\gamma_1$ and $\gamma_2$ in Lemma \ref{lemma-key} are fulfilled, i.e. $\gamma_2\geq0,\gamma_1+\gamma_2<1$ and at least one of the coefficients $\gamma_1,\gamma_2$ is positive. Indeed, since $\alpha\in(0,1)$ and $\beta\geq0$, we know that $\gamma_2\geq0$. If $\gamma_2=0$, then $\beta=0$ and now we have $$\gamma_1=\alpha^2+(1-\alpha)^2+2\alpha(1-\alpha)
	\left(1-2\frac{\sigma_{\min}^2(A)}{\|A\|^2_F}\right)^r>0,$$
	which implies that at least one of the coefficients $\gamma_1,\gamma_2$ is positive.
	The condition $\gamma_1+\gamma_2<1$ holds by the assumption.
	Then apply Lemma \ref{lemma-key} to \eqref{prf-THMm10}, one can get the theorem.
\end{proof}

\begin{remark}
	In \eqref{xie-e-0909}, the inequality $2\langle v,u\rangle\leq \|v\|_2^2+\|u\|^2_2$ was used to estimate the expression $\textcircled{e}$. For future work, one may improve this estimate by using the parameterized Young's inequality $2\langle v,u\rangle\leq \varepsilon\|v\|_2^2+\frac{1}{\varepsilon}\|u\|_2^2$ and optimizing $\varepsilon$ over $\varepsilon>0$.
\end{remark}

\subsection{Proof of Theorems \ref{main-ThmrRDRv2}, \ref{THMfm2}, and \ref{main-ThmrRDRK}}

The following lemmas are useful for our proof.

\begin{lemma}\label{lemma-613}
	Consider the singular value decomposition $A=U\Sigma V^\top$ and let $\{x^k\}_{k=0}^{\infty}$ be the sequences generated by Algorithm \ref{r-mRDRK} or Algorithm \ref{r-RDRK} $(\beta=0)$. Then
	$$ V^\top\mathbb{E}[x^{k+1}-x^*]=\left((1-\alpha+\beta)I+\alpha\left(I-\frac{2\Sigma^\top\Sigma}{\|A\|^2_F}\right)^r\right)
	V^\top\mathbb{E}[x^k-x^*]-\beta V^\top\mathbb{E}[x^{k-1}-x^*].
	$$
\end{lemma}
\begin{proof}
	First, we have
	$$
	\begin{aligned}
		x^{k+1}-x^*&=(1-\alpha) x^{k}+\alpha z^{k}_r+\beta(x^k-x^{k-1})-x^*
		\\
		&=(1-\alpha)(x^k-x^*)+\alpha(z^{k}_r-x^*)+\beta(x^k-x^{k-1}).
	\end{aligned}
	$$
	Taking expectations, we have
	$$
	\begin{aligned}
		\mathbb{E}\big[x^{k+1}-x^*|x^k\big]&=(1-\alpha)( x^k-x^*)+\alpha\mathbb{E}\big[z^{k}_r-x^*|x^k\big]+\beta (x^k-x^{k-1})\\
		&= (1-\alpha)(x^k-x^*)+
		\alpha\left(I-2\frac{A^\top A}{\|A\|^2_F}\right)^r(x^k-x^*)+\beta (x^k-x^{k-1})\\
		&=\left((1-\alpha+\beta)I+\alpha\left(I-2\frac{A^\top A}{\|A\|^2_F}\right)^r\right)(x^k-x^*)-\beta (x^{k-1}-x^*),
	\end{aligned}
	$$
	where the second equality follows from \eqref{prf-equ613-1}.
	Taking the expectations again, we have
	\begin{equation}\label{lemmaeq-613}
		\mathbb{E}[x^{k+1}-x^*]=\left((1-\alpha+\beta)I+\alpha\left(I-2\frac{A^\top A}{\|A\|^2_F}\right)^r\right)\mathbb{E}[x^k-x^*]-\beta \mathbb{E}[x^{k-1}-x^*].
	\end{equation}
	Plugging $A^\top A=V\Sigma^\top\Sigma V^\top$ into \eqref{lemmaeq-613}, and multiplying both sides form the left by $V^\top$, we can get the lemma.
\end{proof}

\begin{lemma}
	\label{lemma-xie-09}
	Suppose the nonzero singular values of $A\in\mathbb{R}^{m\times n}$ are $\sigma_1(A)\geq\sigma_2(A)\geq\ldots\geq \sigma_t(A)=\sigma_{\min}^2(A)$ and  $r\in\mathbb{Z}_{+}$. Then for any $1\leq i\leq t$,
	\begin{equation}\label{lemma-xie-e0908}
		\left(1-2\frac{\sigma_{i}^2(A)}{\|A\|^2_F}\right)^r\leq \left(1-2\frac{\sigma_{\min}^2(A)}{\|A\|^2_F}\right)^r.
	\end{equation}
\end{lemma}
\begin{proof}
	Since
	$$1-2\frac{\sigma_{i}^2(A)}{\|A\|^2_F}\leq 1-2\frac{\sigma_{\min}^2(A)}{\|A\|^2_F},
	$$
	we know that \eqref{lemma-xie-e0908} holds provided that $r$ is odd. Next, we consider the case where $r$ is even. Since
	$$
	\left(1-2\frac{\sigma_{i}^2(A)}{\|A\|^2_F}\right)^2-\left( 1-2\frac{\sigma_{\min}^2(A)}{\|A\|^2_F}\right)^2=\frac{4\left(\sigma_{\min}^2(A)-\sigma_{i}^2(A)\right)\left(
		\|A\|_F^2-\sigma_{\min}^2(A)-\sigma_{i}^2(A)\right)}{\|A\|_{F}^4}\leq0,
	$$
	we know that $\left|1-2\frac{\sigma_{i}^2(A)}{\|A\|^2_F}\right|\leq\left| 1-2\frac{\sigma_{\min}^2(A)}{\|A\|^2_F}\right|$. This implies that \eqref{lemma-xie-e0908} holds for the case where $r$ is even.
\end{proof}

\begin{lemma}[\cite{fillmore1968linear,elaydi1996introduction}]\label{Lemma-relation}
	Consider the second degree linear homogeneous recurrence relation:
	$$
	r^{k+1}=\gamma_{1} r^{k}+\gamma_{2} r^{k-1}
	$$
	with initial conditions $r^{0}, r^{1} \in \mathbb{R}$. Assume that the constant coefficients $\gamma_{1}$ and $\gamma_{2}$ satisfy the inequality $\gamma_{1}^{2}+4 \gamma_{2}<0$ (the roots of the characteristic equation $t^{2}-\gamma_{1} t-$ $\gamma_{2}=0$ are imaginary). Then there are complex constants $c_{0}$ and $c_{1}$ (depending on the initial conditions $r_{0}$ and $r_{1}$) such that:
	$$
	r^{k}=2 M^{k}\left(c_{0} \cos (\theta k)+c_{1} \sin (\theta k)\right)
	$$
	where $M=\left(\sqrt{\frac{\gamma_{1}^{2}}{4}+\frac{\left(-\gamma_{1}^{2}-4 \gamma_{2}\right)}{4}}\right)=\sqrt{-\gamma_{2}}$ and $\theta$ is such that $\gamma_{1}=2 M \cos (\theta)$ and $\sqrt{-\gamma_{1}^{2}-4 \gamma_{2}}=2 M \sin (\theta)$.
\end{lemma}

We first prove Theorem \ref{THMfm2}.
\begin{proof}[Proof of Theorem \ref{THMfm2}]
	Set $s^k:=V^\top\mathbb{E}[x^{k}-x_0^*]$. Then from Lemma \ref{lemma-613} we have
	$$
	s^{k+1}=\left((1-\alpha+\beta)I+\alpha\left(I-\frac{2\Sigma^\top\Sigma}{\|A\|^2_F}\right)^r\right)s^k-\beta s^{k-1},$$
	which can be rewritten in a coordinate  form as follows:
	\begin{equation}\label{prf-THMm2-1}
		s^{k+1}_i=\big((1-\alpha+\beta)+\alpha\big(1-2\sigma^2_i(A)/\|A\|^2_F\big)^r\big)s_i^k-\beta s_i^{k-1}, \ \ \forall \ i=1,2,\ldots,n,
	\end{equation}
	where $s^k_i$ indicates the $i$-th coordinate of $s^k$.
	
	We now consider two cases:
	$\sigma_i(A)=0$ or $\sigma_i(A)>0$.
	
	If $\sigma_i(A)=0$, then \eqref{prf-THMm2-1} takes the form:
	$$
	s^{k+1}_i=(1+\beta)s_i^k-\beta s_i^{k-1}.
	$$
	Since $x^1-x_0^*=x^0-x_0^*=A^\dagger(Ax^0-b)$, we have $s^0_i=s^1_i=v_i^\top A^\dagger(Ax^0-b)=0$.
	So
	$$
	s^k_i=0 \ \ \mbox{for all} \ k\geq0.
	$$
	
	If $\sigma_i(A)>0$. We use Lemma \ref{Lemma-relation} to establish the desired bound. By the selection constraints of $\alpha$, one can verify that $(1-\alpha)+\alpha\big(1-2\sigma_i^2(A)/\|A\|^2_F\big)^r\geq 0$ and since $\beta\geq0$, we have $(1-\alpha)+\beta+\alpha\big(1-2\sigma_i^2(A)/\|A\|^2_F\big)^r\geq 0$ and hence
	$$
	\begin{aligned}
		\gamma_1^2+4\gamma_2&=\big(1-\alpha+\beta+\alpha\big(1-2\sigma_i^2(A)/\|A\|^2_F\big)^r\big)^2-4\beta
		\\
		&\leq\big(1-\alpha+\beta+\alpha\big(1-2\sigma_{\min}^2(A)/\|A\|^2_F\big)^r\big)^2-4\beta\\
		& <0,
	\end{aligned}
	$$
	where the first inequality follows from Lemma \ref{lemma-xie-09} and the last inequality follows from the assumption that
	$\left(1-\sqrt{\alpha\left(1-\left(1-2\sigma_{\min}^2(A)/\|A\|^2_F\right)^r\right)}\right)^2<\beta<1$.
	Using Lemma \ref{Lemma-relation}, the following bound can be deduced
	$$
	s^{k}_{i}=2\left(-\gamma_{2}\right)^{k / 2}\left(c_{0} \cos (\theta k)+c_{1} \sin (\theta k)\right) \leq 2 \beta^{k / 2} p_{i},
	$$
	where $p_{i}$ is a constant depending on the initial conditions (we can simply choose $p_{i}=\left|c_{0}\right|+\left|c_{1}\right|.$)
	Now put the two cases together, for all $k \geq 0$ we have
	$$
	\begin{aligned}
		\|\mathbb{E}[x^{k}-x_0^{*}]\|_{2}^{2} &=\|V^{\top} \mathbb{E}[x^{k}-x_0^{*}]\|_{2}^{2}=\|s^{k}\|^{2}=\sum_{i=1}^{n}(s^{k}_{i})^{2}
		= \sum_{i: \sigma_{i}(A)=0}(s^{k}_{i})^{2}+\sum_{i: \sigma_{i}(A)>0}(s^{k}_{i})^{2}
		\\&= \sum_{i: \sigma_{i}(A)>0}(s^{k}_{i})^{2}
		\leq \sum_{i: \sigma_{i}(A)>0} 4 \beta^{k} p_{i}^{2} =\beta^{k} c,
	\end{aligned}
	$$
	where $c=4 \sum_{i: \lambda_{i}>0} p_{i}^{2}$.
\end{proof}

\begin{proof}[Proof of Theorem \ref{main-ThmrRDRv2}]
	Theorem \ref{main-ThmrRDRv2} can be directly derived from Lemma \ref{lemma-613}. Indeed, using similar arguments as that in the proof of Theorem \ref{THMfm2}, we can get
	$$
	\|\mathbb{E}[x^{k}-x_0^{*}]\|_{2}^{2} = \sum_{i: \sigma_{i}(A)>0}(s^{k}_{i})^{2},
	$$
	where $s_i^{k}=\left((1-\alpha)+\alpha\left(1-2\sigma^2_i(A)/\|A\|^2_F\right)^r\right)s_i^{k-1}, i=1,\ldots,n.$
	It follows from Lemma \ref{lemma-xie-09}, we have
	$$
	\|\mathbb{E}[x^k-x_0^*]\|^2_2\leq \big(1-\alpha+\alpha\left(1-2\sigma^2_{\min}(A)/\|A\|^2_F\right)^r
	\big)^{2k}\|x^0-x_0^*\|^2_2.
	$$
	This completes the proof of the theorem.
\end{proof}

\begin{proof}[Proof of Theorem \ref{main-ThmrRDRK}]
	By Lemma \ref{lemma-613} and using similar arguments as that in the proof of Theorem \ref{THMfm2} and \eqref{prf-THMm2-1}, we know that
	$$
	s^{k+1}_i=\big((1-\alpha)+\alpha\big(1-2\sigma^2_{i}(A)/\|A\|^2_F\big)^r\big)s_i^k, \ \ \forall \ i=1,2,\ldots,n,
	$$
	where
	$s^{k+1}_{i}=\mathop{\mathbb{E}}[\langle x^{k+1}-x^*,v_{i}\rangle]$ and hence the theorem holds.
\end{proof}

\end{document}